\documentclass[12pt,reqno]{amsart}

\usepackage{amssymb,amsmath,graphicx,amsfonts,euscript}
\usepackage{color}

\newtheorem{thm}{Theorem}[section]

\newtheorem{cor}[thm]{Corollary}

\newtheorem{define}[thm]{Definition}

\newtheorem{lemma}[thm]{Lemma}

\def\Dd{\Delta}
\def\pp{\partial}
\def\na{\nabla}

\def\a{\alpha}
\def\la{\langle}
\def\ra{\rangle}
\def\va{\varepsilon}
\def\R{\mathbb{R}}
\def\lr{\langle\na\rangle}

\allowdisplaybreaks

\numberwithin{equation}{section}

\begin{document}
\title[The 2D MHD SYSTEM]{Global small solution to the 2D MHD system with a velocity damping term}

\author[J. Wu, Y. Wu and X. Xu]{Jiahong Wu$^{1}$, Yifei Wu$^{2}$, and  Xiaojing Xu$^{2}$}

\address{$^1$ Department of Mathematics,
Oklahoma State University,
401 Mathematical Sciences,
Stillwater, OK 74078, USA; and
Department of Mathematics,
College of Natural Science,
Chung-Ang University,
Seoul 156-756, Korea}

\email{jiahong@math.okstate.edu}

\address{$^2$School of Mathematical Sciences,
Beijing Normal University and Laboratory of Mathematics and Complex Systems,
Ministry of Education, Beijing 100875, P.R. China}

\email{yifei@bnu.edu.cn;\,\,xjxu@bnu.edu.cn}

\date{\today}

\begin{abstract}
This paper studies the global well-posedness of the incompressible magnetohydrodynamic (MHD)
 system with a velocity damping term. We establish the global existence and uniqueness
of smooth solutions when the initial data is close to an equilibrium state. In addition,
explicit large-time decay rates for various Sobolev norms of
the solutions are also given.
\end{abstract}
\maketitle

\section{Introduction}\label{intro}

\vskip .1in

This paper examines the global (in time) existence and uniqueness of solutions
to the 2D magnetohydrodynamic (MHD) system with a velocity damping term, namely
\begin{equation} \label{e1.1}
\begin{cases}
\partial_t \vec u +\vec u \cdot \nabla \vec u + \vec u + \nabla P= -\nabla \cdot (\nabla \phi\otimes\nabla\phi),\quad (t,\,x,\,y)\in \mathbb{R}_+\times\mathbb{R}\times\mathbb{R},\\
\partial_t \phi +\vec u\cdot \nabla \phi=0, \\
\nabla \cdot \vec u = 0,\\
\vec u|_{t=1}=\vec u_0(x,\,y),\quad\phi|_{t=1}=\phi_0(x,\,y),
\end{cases}
\end{equation}
where $\vec u=(u,\,v)$ represents the 2D velocity field, $P$ the pressure and $\phi$
the magnetic stream function, and $\nabla \phi\otimes\nabla\phi$ denotes the tensor product.
\eqref{e1.1} is formally equivalent to the
2D MHD equations given by
\begin{equation} \label{e1.11}
\begin{cases}
\partial_t \vec u +\vec u \cdot \nabla \vec u + \vec u + \nabla P= -\frac12\nabla\big(|\vec b|^2\big)+\vec b\cdot \nabla \vec b,\\
\partial_t \vec b +\vec u\cdot \nabla \vec b=\vec b\cdot \nabla \vec u,  \\
\nabla \cdot \vec b = \nabla \cdot \vec u = 0.
\end{cases}
\end{equation}
In fact, $\nabla \cdot \vec b=0$ implies that
$\vec b=\nabla^\perp \phi\equiv \big(\partial_y\phi,-\partial_x\phi\big)$
for a scalar function $\phi$ and, with this substitution,
(\ref{e1.11}) is reduced to \eqref{e1.1}. The MHD equations, modeling
electrically conducting fluid in the presence of a magnetic field,
consist essentially of the interaction between the fluid velocity and
the magnetic field. Electric currents induced in the fluid as a result of
its motion modify the field; at the same time their flow in
the magnetic field leads to mechanical forces which modify
the motion. The MHD equations underly many phenomena such as
the geomagnetic dynamo in geophysics and  solar winds and
solar flares in astrophysics (see, e.g., \cite{Bis,Dav,Pri}).

\vskip .1in
Mathematically the MHD equations can be extremely difficult to analyze due to
the nonlinear coupling between the forced Navier-Stokes equations and the
induction equation. In fact, it remains an outstanding open problem whether
solutions to the 2D MHD equations
\begin{equation} \label{idealmhd}
\begin{cases}
\partial_t \vec u +\vec u \cdot \nabla \vec u + \nabla P
= -\frac12\nabla\big(|\vec b|^2\big)+\vec b\cdot \nabla \vec b,\\
\partial_t \vec b +\vec u\cdot \nabla \vec b=\vec b\cdot \nabla \vec u,  \\
\nabla \cdot \vec b = \nabla \cdot \vec u = 0
\end{cases}
\end{equation}
exist for all time or they blow up in a finite time. One main difficulty is the
lack of global (in time) bounds for the Sobolev norms of the solutions. Adding a velocity
damping term does not appear to be sufficient to overcome this difficulty and our aim here
is at small global smooth solutions. Since the equation of $\phi$ in \eqref{e1.1} is a
transport equation without any damping or dissipation, it is a very involved problem to establish the
global well-posedness of \eqref{e1.1} even under the assumption that the initial data is small.

\vskip .1in
The global regularity problem on the 2D MHD equations with partial
dissipation or partial damping has attracted considerable interests
in the last few years and progress has been made for some cases. The anisotropic 2D MHD
equations with horizontal dissipation and vertical magnetic diffusion were recently examined
by Cao and Wu and shown to possess global classical solutions for any sufficiently
smooth data \cite{CaoWu}. Advances have also been achieved for the case when the dissipation
and the magnetic diffusion are both in the horizontal direction (\cite{CaoReWu,CaoReWuZ}).
F. Lin, L. Xu and P. Zhang recently studied the MHD equations with the Laplacian dissipation in the
velocity equation but without magnetic diffusion and, remarkably, they were able to establish the global existence
of small solutions after translating
the magnetic field by a constant vector (\cite{LinZhang1,LinZhang2,XuZhang}). Their approach
reformulates the system in Lagrangian coordinates and estimates the Lagrangian velocity through
the anisotropic Littlewood-Paley theory and anisotropic Besov space techniques. The partial
dissipation case when only the magnetic diffusion is present has also been examined and the global
$H^1$ weak solutions have been established (see, e.g., \cite{CaoWu,LZ}). In addition, if we increase
the magnetic diffusion from the Laplacian operator to the fractional Laplacian operator $(-\Delta)^\beta$
with $\beta>1$, then the resulting MHD equations do have global regular solutions (\cite{CaoWuYuan,JiuZhao2}). Many more
recent results on the MHD equations with partial or fractional dissipation can be bound in the references
\cite{CMZ1,CMZ2,HeXin1,HeXin2,JiuZhao,Lei,TrYu,Wang,Wu2,Wu3,Wu4,YuanBai,Yama1,Yama2,Zhou}.

\vskip .1in
The contribution of this paper is the global existence and uniqueness of
solutions of \eqref{e1.1} with sufficiently smooth initial data $(u_0,\phi_0)$ close to
the equilibrium state $(0,y)$.
This work is partially inspired by \cite{LinZhang1}. Our approach here exploits the time decay
properties of the solution kernels to a linear differential equation which, with suitable
nonlinear forcing terms,
governs the translated version of \eqref{e1.1}. We now give a more precise account of our ideas. Setting
$$
 \phi=\psi+y
$$
converts \eqref{e1.1} into the following equivalent system of equations for $(u,v, \psi)$,
\begin{equation}\label{e1.2}
\begin{cases}
\pp_t u+u\,\pp_x u+v\pp_yu+u+\pp_x \tilde P=-\Dd\psi\pp_x\psi,\\
\pp_tv+u\,\pp_x v+v\pp_yv+v+\pp_y \tilde P=-\Dd\psi-\Dd\psi\pp_y\psi,\\
\pp_t\psi+u\pp_x\psi+v\pp_y\psi+v=0,\\
\pp_xu+\pp_yv=0,
\end{cases}
\end{equation}
where $\tilde P=P + \frac12 |\na \phi|^2$. Applying
$\na \cdot \vec u=0$ to eliminate the pressure term yields
\begin{align}
\label{u1}
&\pp_t u+u-\pp_{xy}\psi=\Pi_1,\\
\label{v1}
&\pp_t v+v{+}\pp_{xx}\psi=\Pi_2,\\
\label{p1}
&\pp_t\psi+u\pp_x\psi+v\pp_y\psi+v=0,
\end{align}
where
\begin{align}
&\Pi_1=-\vec u\cdot\na u+{\pp_x }{\Delta}^{-1}\na\cdot(\vec u\cdot\na \vec u)
-\Dd\psi\,\pp_x\psi +{\pp_x }{\Delta}^{-1}\na\cdot(\Dd \psi\,\na \psi),
\label{N1term}\\
&\Pi_2=-\vec u\cdot\na v+{\pp_y}{\Delta}^{-1}\na\cdot(\vec u\cdot\na \vec u)
-\Dd\psi\,\pp_y\psi +{\pp_y }{\Delta}^{-1}\na\cdot(\Dd \psi\,\na \psi). \label{N2term}
\end{align}
Taking the time derivative on the equations \eqref{u1}--\eqref{p1}, we obtain
\begin{equation}\label{e1.6}
\begin{cases}
\pp_{tt} u+\pp_t u-\pp_{xx} u=F_1,
\\
\pp_{tt} v+\pp_t v-\pp_{xx} v=F_2,
\\
\pp_{tt} \psi+\pp_t \psi-\pp_{xx} \psi=F_0,
\\
\vec u|_{t=1}=\vec u_0(x,\,y),\quad \vec u_t|_{t=1}=\vec u_1(x,\,y),
\\
\psi|_{t=1}=\psi_0(x,\,y),\quad \psi_t|_{t=1}=\psi_1(x,\,y),
\end{cases}
\end{equation}
where $\vec u_1=(u_1(x,\,y),\,v_1(x,\,y)),\,\,\psi_0=\phi_0-y$,  and
\begin{align*}
&u_1=(-u+\pp_{xy}\psi+\Pi_1)|_{t=1},\\
&v_1=(-v-\pp_{xx}\psi+\Pi_2)|_{t=1},\\
&\psi_1=(-u\pp_x\psi-v\pp_y\psi-v)|_{t=1},
\end{align*}
and
\begin{align}
&F_0=-\vec u\cdot \na\psi-\pp_t(\vec u\cdot \na \psi){-\Pi_2},
\label{Def:F0}\\
&F_1=\pp_t\Pi_1-\pp_{xy}(\vec u\cdot \na \psi),\label{Def:F1}\\
&F_2=\pp_t\Pi_2+\pp_{xx}(\vec u\cdot \na \psi).\label{Def:F2}
\end{align}
The structure of the linear part in (\ref{e1.6}) plays a crucial role in ensuring the
global existence of small solutions. In fact, the solution kernels of the linear equation
decay in time in suitable spatial functional settings. Let us be more accurate. As detailed
in Section \ref{preliminary},
the solution of the linear equation
\begin{align*}
\partial_{tt}\Phi+\partial_t \Phi-\partial_{xx}\Phi=0
\end{align*}
with the initial data
$$
\Phi(0,x,y)=\Phi_0(x,y),  \,\,\, \Phi_t(0,x,y)=\Phi_1(x,y)
$$
can be written as
$$
{\Phi}(t,x,y)=K_0(t,\partial_x){\Phi}_0+K_1(t,\partial_x)\big(\frac12{\Phi}_0+{\Phi}_1\big),
$$
where the solution operators $K_1$ and $K_2$ are explicitly derived in Section \ref{preliminary}.
By Duhamel's Principle, the solution of the inhomogeneous equation
\begin{align*}
\partial_{tt}\Phi+\partial_t \Phi-\partial_{xx}\Phi=F,
\end{align*}
with initial data $\Phi(1,x)=\Phi_0,\, \partial_t\Phi(1,x)=\Phi_1$ is given by
\begin{align}
\Phi(t,x,y)=&K_0(t,\partial_x)\Phi_0+K_1(t,\partial_x)\big(\frac12\Phi_0+\Phi_1\big) \nonumber\\
&+\int_1^t K_1(t-s,\partial_x)F(s,x,y)\,ds.\label{e1.6int}
\end{align}
By letting $\Phi=(u,v,\psi)$ and $F=(F_0,F_1,F_2)$, (\ref{e1.6int}) gives an integral representation
of (\ref{e1.6}).
Thanks to the time decay properties of $K_1$ and $K_2$ (established in Section \ref{preliminary}),
the nonlinear parts in (\ref{e1.6}) remain small and the solution map is a contraction for all time.
More details will be unfolded in the
subsequent sections.

\vskip .1in
To state our main result, we introduce the functional settings.
Let $X_0$ be the Banach space defined by the norm
\begin{align*}
\|(\vec u_0,\,\psi_0)\|_{X_0}=\|\la\na\ra^N(\vec u_0,\,\nabla \psi_0)\|_{L^2_{xy}}
+\|\lr^{6+}(\vec u_0,\,\psi_0)\|_{L^1_{xy}}+\|\lr^{6+}(\vec u_1,\,\psi_1)\|_{L^1_{xy}}.
\end{align*}
where $\la\na\ra=(I-\Delta)^{\frac12}$  and $a+$ denotes $a+\epsilon$ for
any small $\epsilon$. For notational convenience, we also write
$$
\|f\|_q = \|f\|_{L^q_{xy}}, \qquad 1\le q\le \infty.
$$
Now we define our working space as $X$ with its norm given by
\begin{equation}\label{X}
\begin{aligned}
\|(\vec u,\,\psi)\|_{X} = \sup_{t\geq 1}&\left\{t^{-\varepsilon}\|\la\na\ra^N(\vec u(t),\,
\nabla \psi(t))\|_2+t^{\frac14}\|\la\na\ra^3\psi\|_2\right.\\
&+t^\frac32\|\la\na\ra\pp_{xx}\psi\|_\infty+ t^\frac54\|\la\na\ra^3\pp_{xx}\psi\|_2
+ t^\frac32\|\pp_{xxx}\psi\|_2\\
&\left. + t^\frac32\|\pp_{t}\vec{u}\|_\infty + t^\frac54\|\la \na\ra\pp_{t}\vec{u}\|_2+t\|\la\na\ra\pp_{x}\vec u\|_\infty
+t^\frac32\|\partial_x\partial_t v\|_2\right\}.
\end{aligned}
\end{equation}

Here $N$ is a big positive integer and $\varepsilon>0$ is a small parameter. For the sake of the clarity
of our presentation, we intentionally avoid the tedious calculations needed for providing an accurate range
of $N$. However, sufficiently large $N$ and small $\varepsilon>0$, say $N=20$ and $\varepsilon=0.01$,
would serve our purpose.

\vskip .1in
Our main result can then be stated as follows. We use $A \lesssim B$ or $B  \gtrsim A$ to denote the statement
that $A\le C B$ for some absolute constant $C>0$.
\begin{thm} \label{th1}
Let $\psi=\phi-y$ and $\psi_0=\phi_0-y$. Then there exists a small constant $\varepsilon_0>0$ such that,
if the initial data $(\vec u_0,\,\phi_0)$ satisfying $\|(\vec u_0,\,\psi_0)\|_{X_0}\le\varepsilon_0 $,
then there exists a unique global solution $(u,\,v,\,\phi,\,P)$ to the system \eqref{e1.1} with
$$
(u,\,v,\,\phi)\in X,\qquad P\in C\big([1,\infty);H^N(\R^2)\big).
$$
Moreover, the following decay estimates hold
$$
\|u(t)\|_{L^\infty_{xy}}\lesssim \varepsilon_0 t^{-1};\quad \|v(t)\|_{L^\infty_{xy}}\lesssim \varepsilon_0 t^{-\frac32};\quad
\|\psi(t)\|_{L^\infty_{xy}}\lesssim \varepsilon_0 t^{-\frac12}; \quad \|P(t)\|_{L^\infty_{xy}}\lesssim \varepsilon_0 t^{-\frac12}.
$$
\end{thm}

\vskip .1in
The proof of Theorem \ref{th1} relies on the following lemma, which can be deduced from a standard
continuity argument (see, e.g.,  Theorem 4 in  \cite{BT}).

\begin{lemma}\label{con}
Assume the initial data $(\vec u_0,\,\psi_0)\in X_0$. Suppose that $(\vec u,\,\psi)$ given by (\ref{e1.6int})
namely the integral representation of (\ref{e1.6}), satisfies
\begin{align}\label{e1.7}
\|(\vec u,\, \psi)\|_{X}\lesssim \|(\vec u_0,\, \psi_0)\|_{X_0} + Q(\|\vec u,\, \psi)\|_X),
\end{align}
where $Q(a)\geq C a^\beta$ for $a\lesssim 1$ and $\beta>1$.
Then there exists $r_0>0$ such that, if
\begin{align*}
\|(\vec u_0,\, \psi_0)\|_{X_0}\lesssim r_0,
\end{align*}
then (\ref{e1.6}) has a unique global solution $(\vec u,\,\psi)\in X$ and
\begin{align*}
\|(\vec u,\, \psi)\|_{X}\lesssim 2r_0.
\end{align*}
\end{lemma}

In addition, to facilitate the proof, we introduce an auxiliary functional space, in which more terms with explicit time decay estimates are included. Let
\begin{align*}
\|(\vec u,\,\psi)\|_{Y}=&\|(\vec u,\,\psi)\|_X+\sup_{t\geq 1}\left\{t^\frac12\|\la\na\ra^2\psi\|_\infty+t^\frac34\|\la\na\ra\pp_{x}\psi\|_2+\, t\|\pp_x\la\na\ra^3\psi\|_\infty
\right.\\
& \left.+t\|u\|_\infty +t^\frac32\|v\|_\infty+t^\frac34\|u\|_2+
t^\frac54\|\pp_{x}u\|_2+ t^\frac54\|v\|_2\right\}.
\end{align*}
Roughly speaking, the decay rates of the extra terms in the $Y$-norm obey the following rules:
$$
L^\infty\sim t^{-\frac12};\quad L^2\sim t^{-\frac14};\quad \partial_x\sim t^{-\frac12};\quad \partial_t\sim \partial_{xx};\quad u\sim \partial_{x}\sim t^{-\frac12};\quad v\sim \partial_{xx}\sim t^{-1}.
$$
As we show in Section \ref{sec3}, the norms $\|(\vec u,\, \psi)\|_{Y}$ and $\|(\vec u,\, \psi)\|_{X}$ are related
through the following lemma.
\begin{lemma}\label{yl1}
Let the spaces $X, \,Y$ and their norms be defined as above. Then
\begin{align*}
\|(\vec u,\, \psi)\|_{Y}\lesssim \|(\vec u,\, \psi)\|_{X} +  Q(\|(\vec u,\, \psi)\|_X).
\end{align*}
\end{lemma}

As a consequence of Lemma \ref{yl1}, to prove  \eqref{e1.7}, it is enough to verify
\begin{align}\label{con1}
\|(\vec u,\, \psi)\|_{X}\lesssim \|(\vec u_0,\, \psi_0)\|_{X_0}+Q\big(\|(\vec u,\, \psi)\|_Y\big).
\end{align}
Therefore, the proof of Theorem \ref{th1} is then reduced to establishing (\ref{con1})
and our main effort is devoted to achieving this goal. This is a long process and involves
several major components. The first consists of crucial and sharp decay estimates and
four tool lemmas on the kernels. The second is the bounds on the nonlinearities $F_0$, $F_1$ and $F_2$
by treating differently the lower and higher frequencies of their terms. The third involves
the estimates of each member in the $X$-norm in \eqref{X} through the first two components.

\vskip .1in
The rest of the paper is divided into six sections and an appendix.  The second section derives the
solution kernel of the linear equation and represents the solution of (\ref{e1.6}) in an integral
form through the Duhamel formula. Crucial decay estimates for the solution kernels and four tool lemmas
to be used repeatedly are also presented in this section. Section \ref{sec3} proves Lemma \ref{yl1}.
The rest of the sections are devoted to proving (\ref{con1}). Section \ref{sec:energy} bounds
 $t^{-\varepsilon}\|\la\na\ra^N(\vec u(t),\,\nabla \psi(t))\|_2$ through energy estimates,
which give a control of the first term in the definition of the norm of $X$. Section \ref{sec:non}
provides suitable estimates for the nonlinear terms $F_0$, $F_1$ and $F_2$. To obtain these estimates,
we decompose the terms involved to high and low frequencies and apply the results from the second section and
an inequality involving the Riesz transform (Lemma \ref{lem:Riesz-bound}). With the estimates for
$F_0$, $F_1$ and $F_2$ at our disposal,
Sections \ref{s2} and \ref{s3} continue the proof of (\ref{con1}) by repeatedly applying the tool
lemmas in the second section and the estimates for
$F_0$, $F_1$ and $F_2$. The appendix serves four purposes.
It gives an explicit representation of $\Pi_1$ and $\Pi_2$. The key point of this representation is that each term
is written in a way that it possesses as many directives in the $x$-direction as possible. As seen from
Section \ref{preliminary}, the more $x-$derivatives a term has, the faster
it decays in time. This point has played an important role in the estimates of $F_0$, $F_1$ and $F_2$ in
Section \ref{sec:non}. It also provides the proofs of Lemma \ref{lem:Riesz-bound} and Lemma \ref{lem:RR}.
Finally, the properties of the pressure are also given here.

\vskip .4in
\section{Preliminary}
\label{preliminary}

This section is divided into two subsections. The first subsection derives the integral formulation
(\ref{e1.6int}) with explicit representations for $K_1$ and $K_2$. In addition, key decay estimates
for $K_1$ and $K_2$ are also obtained here. The second subsection proves several tool lemmas to be
used in the proof of Theorem \ref{th1}.

\subsection{Linear operators}
We consider the linear equation
\begin{align}\label{Lineareqs}
\partial_{tt}\Phi+\partial_t \Phi-\partial_{xx}\Phi=0
\end{align}
with the initial data
$$
\Phi(0,x,y)=\Phi_0(x,y),\,\,\Phi_t(0,x,y)=\Phi_1(x,y).
$$
Taking the Fourier transform of \eqref{Lineareqs} yields
\begin{align}\label{Lineareqs1}
\partial_{tt}\hat{\Phi}+\partial_t \hat{\Phi}+\xi^2\hat {\Phi}=0,
\end{align}
where the  Fourier transform $\hat \Phi$ is defined as
$$
\hat {\Phi}(t,\xi,\eta)=\int_{\R^2} e^{{-}ix\xi{-}iy\eta}{\Phi}(t,x,y)\,dx dy.
$$
Solving \eqref{Lineareqs1} by a simple ODE theory, we have
\begin{align*}
\hat{\Phi}(t,\xi,\eta)=&
\frac12\Big(e^{\big(-\frac12+\sqrt{\frac14-\xi^2}\big)t}
+e^{\big(-\frac12-\sqrt{\frac14-\xi^2}\big)t}\Big)\widehat{{\Phi}_0}(\xi,\eta)\\
&
+
\frac1{2\sqrt{\frac14-\xi^2}}\Big(e^{(-\frac12+\sqrt{\frac14-\xi^2})t}
-e^{(-\frac12-\sqrt{\frac14-\xi^2})t}\Big)\Big(\frac12\widehat{{\Phi}_0}(\xi,\eta)+\widehat{{\Phi}_1}(\xi,\eta)\Big).
\end{align*}
\begin{define}\label{k01}
Let the operators $K_0(t,\partial_x)$ and $K_1(t,\partial_x)$ be defined as
$$
\widehat{K_0(t,\pp_x) f}(t,\xi,\eta)=\frac12\Big(e^{\big(-\frac12+\sqrt{\frac14-\xi^2}\big)t}
+e^{\big(-\frac12-\sqrt{\frac14-\xi^2}\big)t}\Big)
\hat{f}(t,\xi,\eta);
$$
and
$$
\widehat{K_1(t,\pp_x) f}(t,\xi,\eta)=
\frac1{2\sqrt{\frac14-\xi^2}}\Big(e^{\big(-\frac12+\sqrt{\frac14-\xi^2}\big)t}
-e^{\big(-\frac12-\sqrt{\frac14-\xi^2}\big)t}\Big)\hat{f}(t,\xi,\eta),
$$
where $\sqrt{-1}=i$.
\end{define}
By Definition \ref{k01}, the solution ${\Phi}$ of the equation \eqref{Lineareqs} is written as
$$
{\Phi}(t,x,y)=K_0(t,\partial_x){\Phi}_0+K_1(t,\partial_x)\big(\frac12{\Phi}_0+{\Phi}_1\big).
$$
Moreover, consider the inhomogeneous equation,
\begin{align}\label{Inhomoeqs0}
\partial_{tt}\Phi+\partial_t \Phi-\partial_{xx}\Phi=F,
\end{align}
with initial data $\Phi(1,x)=\Phi_0,\, \partial_t\Phi(1,x)=\Phi_1$. Then we have the following standard Duhamel formula,
\begin{align}\label{Inhomoeqs}
\Phi(t,x,y)=K_0(t,\partial_x)\Phi_0+K_1(t,\partial_x)\big(\frac12\Phi_0+\Phi_1\big)+\int_1^t K_1(t-s,\partial_x)F(s,x,y)\,ds.
\end{align}

In the following, we present some decay estimates on $K_0,\,K_1$.

\begin{lemma}\label{lem:estK0K1}
Let $K_0, K_1$ be defined in Definition \ref{k01}. Then for any $\alpha\ge 0,0\le \beta\le 2$, $1\le q \le \infty,\;i=0,\,1$,
\begin{itemize}
\item[1)] $\big\||\xi|^\alpha\widehat {K_i}(t,\cdot)\big\|_{L^q_\xi(|\xi|\le \frac12)}\lesssim {\langle t\rangle}^{-\frac12(\frac1q+\alpha)};$
\item[2)] $\big\||\xi|^{-\beta}\partial_t\widehat {K_i}(t,\cdot)\big\|_{L^q_\xi(|\xi|\le \frac12)}\lesssim{\langle t\rangle}^{-1-\frac12(\frac1q-\beta)};$
\item[3)] $\big|\widehat {K_i}(t,\xi)\big|\lesssim e^{-\frac12t},$ for any $|\xi|\ge\frac12, \;i=0,\,1;$
\item[4)]  $\big|\langle \xi \rangle^{-1}\pp_t\widehat{K_0}(t, \xi)\big|, \big|\partial_t\widehat {K_1}(t,\xi)\big|\lesssim e^{-\frac12t},$ for any $|\xi|\ge\frac12$.
\end{itemize}
\end{lemma}
\begin{proof} Since the decay properties of the operators $K_0,K_1$ are distinct between low
and high frequencies, we will split its frequency into the following two parts:
$$
|\xi|\le \frac12;\qquad |\xi|> \frac12.
$$
In the following, we will analyze the two parts separately.

\vskip .1in
Case I:  $|\xi|\le \frac12.$

According to the expressions of $K_0,\,K_1$ as follows
$$
\widehat{K_0}(t,\xi)=\frac12\Big(e^{\big(-\frac12+\sqrt{\frac14-\xi^2}\big)t}+e^{\big(-\frac12-\sqrt{\frac14-\xi^2}\big)t}\Big)
$$
and
$$
\widehat{K_1}(t,\xi)=\frac1{2\sqrt{\frac14-\xi^2}}\Big(e^{(-\frac12+\sqrt{\frac14-\xi^2})t}-e^{(-\frac12-\sqrt{\frac14-\xi^2})t}\Big),
$$
we obtain, for $|\xi|\le \frac12$,
$$
-\frac12+\sqrt{\frac14-\xi^2}=\frac{-\xi^2}{\frac12+\sqrt{\frac14-\xi^2}}\le -\xi^2;\quad
-\frac12-\sqrt{\frac14-\xi^2}\le -\xi^2.
$$
Then we get
$$
0\le \widehat{K_0}(t,\xi)\le e^{-t\xi^2}.
$$
Therefore,
\begin{align*}
|\xi|^\alpha\widehat {K_0}(t,\xi)&\le |\xi|^\alpha e^{-t\xi^2} \lesssim \langle t\rangle^{-\frac\alpha2};\\
\big\||\xi|^\alpha\widehat {K_0}(t,\xi)\big\|_{L^1_\xi(|\xi|\le \frac12)}&\le \int_{|\xi|\le \frac12} |\xi|^\alpha\widehat {K_0}(t,\xi)\,d\xi\le \int_{|\xi|\le \frac12} |\xi|^\alpha e^{-t\xi^2}\,d\xi
\lesssim \langle t\rangle^{-\frac12(1+\alpha)}.
\end{align*}
So the estimate 1) of Lemma \ref{lem:estK0K1} for $K_0$ follows from interpolation.
Moreover,
\begin{align}
\dot{K_0}(t)=&-\frac12K_0(t)+\big(\frac14+\partial_{xx}\big) K_1(t);\label{dotK0}\\
\dot{K_1}(t)=&K_0(t)-\frac12 K_1(t).\label{dotK1}
\end{align}
Thus,
\begin{align}\label{dotK11}
\partial_t \widehat{K_0}(t,\xi)=&-\frac12\widehat{K_0}(t,\xi)+\big(\frac14-\xi^2\big) \widehat{K_1}(t,\xi)\notag\\
=&\frac14\big(\sqrt{1-4\xi^2}-1\big)e^{(-\frac12+\sqrt{\frac14-\xi^2})t}-
\frac14\big(\sqrt{1-4\xi^2}+1\big)e^{(-\frac12-\sqrt{\frac14-\xi^2})t}.
\end{align}
So we have for any $0\le \beta \le 2$,
$$
\big||\xi|^{-\beta}\partial_t \widehat{K_0}(t,\xi)\big|\lesssim |\xi|^{2-\beta}e^{-t\xi^2}.
$$
Rough speaking, $\partial_t \widehat{K_0}(t,\xi)\sim \xi^2\widehat{K_0}(t,\xi)$, and thus we have
$$
\big\||\xi|^{-\beta}\partial_t\widehat {K_0}(t,\cdot)\big\|_{L^q_\xi(|\xi|\le \frac12)}\lesssim \langle t\rangle^{-1-\frac 12(\frac1{q}-\beta)}.
$$
Similarly, we can obtain
$$
0\le \widehat{K_1}(t,\xi)\le e^{-t\xi^2};\quad \big|\partial_t \widehat{K_1}(t,\xi)\big|\lesssim \xi^2e^{-t\xi^2}.
$$
Thus for $K_1$, we have the same estimates as $K_0$.
Hence, we deduce the estimates 1) and 2) of Lemma \ref{lem:estK0K1}.

\vskip .1in
Case II:   $|\xi|>\frac12.$

For this case, the expressions of  $K_0,\,K_1$ can be written as
$$
\widehat{K_0}(t,\xi)=\frac12\Big(e^{\big(-\frac12+i\sqrt{\xi^2-\frac14}\big)t}
+e^{\big(-\frac12-i\sqrt{\xi^2-\frac14}\big)t}\Big)
$$
and
$$
\widehat{K_1}(t,\xi)=\frac1{2i\sqrt{\xi^2-\frac14}}
\Big(e^{(-\frac12+i\sqrt{\xi^2-\frac14})t}-e^{(-\frac12-i\sqrt{\xi^2-\frac14})t}\Big).
$$
By virtue of the expression of $K_0,\,K_1$, then we get, for any $|\xi|>\frac12$,
$$
\big|\widehat {K_0}(t,\xi)\big|, \,\big|\widehat {K_1}(t,\xi)\big|\lesssim e^{-\frac12t},
$$
Further, by \eqref{dotK1} and \eqref{dotK11}, we also have
$$
\big|\langle \xi \rangle^{-1}\pp_t\widehat{K_0}(t, \xi)\big|, \big|\partial_t\widehat {K_1}(t,\xi)\big|\lesssim e^{-\frac12t}.
$$
Hence we complete the proof of Lemma \ref{lem:estK0K1}.
\end{proof}

\subsection{Tool lemmas}
To prove \eqref{con1}, we need several lemmas. In the following, $\mathcal{S}$
denotes the Schwartz class and $\|f\|_{L^p_x L^q_y} \equiv \|\|f\|_{L^q_y}\|_{L^p_x}$.

\begin{lemma}\label{lem:tool1}
Let $K(t,\partial_x)$ denote a Fourier multiplier operator with
$$
\big\|\widehat{K}(t,\xi)\big\|_{L^1_\xi} < \infty.
$$
Then, for any space-time Schwartz function $f$,
\begin{align}\label{tool1}
\big\|K(t,\partial_x)f\big\|_{L^\infty_{xy}}\lesssim \big\|\widehat{K}(t,\xi)\big\|_{L^1_\xi}\big\|\partial_y f\big\|_{L^1_{xy}}.
\end{align}
\end{lemma}
\begin{proof}
For any $g\in \mathcal{S}(\R)$, we have
\begin{align}\label{infty-1}
\|g\|_{L^\infty(\R)}\lesssim \|g'\|_{L^1(\R)}.
\end{align}
Then, by the inequality above and Young's inequality, we have
\begin{align*}
\big\|K(t,\partial_x)f\big\|_{L^\infty_{xy}}
&\lesssim
\big\|K(t,\partial_x)\partial_y f\big\|_{L^\infty_{x}L^1_y}\\
&\lesssim
\big\|K(t,\partial_x)\partial_y f\big\|_{L^1_yL^\infty_{x}}\\
&\lesssim
\Big\|\big\|\widehat{K}(t,\xi)\mathcal F_\xi\big(\partial_y f)(t,\xi,y)\big\|_{L^1_{\xi}}\Big\|_{L^1_{y}}\\
&\lesssim
\big\|\widehat{K}(t,\xi)\big\|_{L^1_\xi}\left\|\big\|\mathcal F_\xi\big(\partial_y f)(t,\xi,y)\big\|_{L^\infty_{\xi}}\right\|_{L^1_{y}}\\
&\lesssim
\big\|\widehat{K}(t,\xi)\big\|_{L^1_\xi}\left\|\big\|\partial_y f(t,x,y)\big\|_{L^1_{x}}\right\|_{L^1_{y}}\\
&=
 \big\|\widehat{K}(t,\xi)\big\|_{L^1_\xi}\big\|\partial_y f\big\|_{L^1_{xy}}.
\end{align*}
This proves Lemma \ref{lem:tool1}.
\end{proof}

\begin{lemma}\label{lem:tool2}
Assume that  $\big\|\widehat{K}(t,\cdot)\big\|_{L^\infty}$  is bounded. Then, for any space-time Schwartz function $f$, and any $\epsilon>0$,
\begin{align}\label{tool2}
\big\|K(t,\partial_x)f\big\|_{L^\infty_{xy}}\lesssim \big\|\widehat{K}(t,\xi)\big\|_{L^\infty_\xi}
\big\|\partial_y \langle\nabla\rangle^{1+\epsilon} f\big\|_{L^1_{xy}}.
\end{align}
\end{lemma}
\begin{proof}
By \eqref{infty-1}, Sobolev's inequality and Plancherel's identity, we have
\begin{align*}
\big\|K(t,\partial_x)f\big\|_{L^\infty_{xy}}
&\lesssim
\big\|K(t,\partial_x)\partial_y f\big\|_{L^\infty_{x}L^1_y}\\
&\lesssim
\big\|K(t,\partial_x)\partial_y \langle\nabla\rangle^{\frac12+\frac\epsilon2}f\big\|_{L^1_yL^2_{x}}\\
&=
\Big\|\big\|\widehat{K}(t,\xi)\mathcal F_\xi\big(\partial_y \langle\nabla\rangle^{\frac12+\frac\epsilon2} f)(t,\xi,y)\big\|_{L^2_{\xi}}\Big\|_{L^1_{y}}\\
&\lesssim
\big\|\widehat{K}(t,\xi)\big\|_{L^\infty_\xi}\left\|\big\|\mathcal F_\xi\big(\partial_y\langle\nabla\rangle^{\frac12+\frac\epsilon2} f)(t,\xi,y)\big\|_{L^2_{\xi}}\right\|_{L^1_{y}}\\
&=
\big\|\widehat{K}(t,\xi)\big\|_{L^\infty_\xi}\left\|\big\|\partial_y\langle\nabla\rangle^{\frac12+\frac\epsilon2} f(t,x,y)\big\|_{L^2_{x}}\right\|_{L^1_{y}}\\
&\lesssim
\big\|\widehat{K}(t,\xi)\big\|_{L^\infty_\xi}\big\|\partial_y \langle\nabla\rangle^{1+\epsilon} f\big\|_{L^1_{xy}}.
\end{align*}
This proves Lemma \ref{lem:tool2}.
\end{proof}

As a special consequence of Lemmas \ref{lem:tool1} and \ref{lem:tool2}, we have the following corollary.
\begin{cor} \label{cor:tool12}
Let $K(t,\partial_x)$ be a Fourier multiplier operator satisfying
$$
\big\|\widehat{\partial_x^\alpha K}(t,\xi)\big\|_{L^1_\xi(|\xi|\le\frac12)} <\infty, \quad
\big\|\widehat{K}(t,\xi)\big\|_{L^\infty_\xi(|\xi|\ge\frac12)} <\infty, \quad \alpha\ge 0.
$$
Then, for any space-time Schwartz function $f$,
\begin{align}\label{es:infty}
\big\|\partial_x^\alpha K(t,\partial_x)f\big\|_{L^\infty_{xy}}\lesssim &\Big(\big\|\widehat{\partial_x^\alpha K}(t,\xi)\big\|_{L^1_\xi(|\xi|\le\frac12)}
+\big\|\widehat{K}(t,\xi)\big\|_{L^\infty_\xi(|\xi|\ge\frac12)}\Big) \nonumber \\
& \times \big\|\langle\nabla\rangle^{\alpha+1+\epsilon}\partial_y f\big\|_{L^1_{xy}}.
\end{align}
\end{cor}

\begin{lemma}\label{lem:tool3}
Assume that  $\big\|\widehat{K}(t,\cdot)\big\|_{L^2}$ is bounded. Then, for any space-time Schwartz function $f$, and any $\epsilon>0$,
\begin{align}\label{tool3}
\big\|K(t,\partial_x)f\big\|_{L^2_{xy}}\lesssim \big\|\widehat{K}(t,\xi)\big\|_{L^2_\xi}\big\||\nabla|^{\frac12-\epsilon}\langle\nabla\rangle^{2\epsilon} f\big\|_{L^1_{xy}}.
\end{align}
\end{lemma}
\begin{proof}
By a similar manner, we have
\begin{align*}
\big\|K(t,\partial_x)f\big\|_{L^2_{xy}}
&=\Big\|\big\|\widehat{K}(t,\xi)\mathcal F_\xi f(t,\xi,y)\big\|_{L^2_{\xi}}\Big\|_{L^2_{y}}\\
&\lesssim
\big\|\widehat{K}(t,\xi)\big\|_{L^2_\xi}\left\|\big\|\mathcal F_\xi f(t,\xi,y)\big\|_{L^\infty_{\xi}}\right\|_{L^2_{y}}\\
&\lesssim
\big\|\widehat{K}(t,\xi)\big\|_{L^2_\xi}\left\|\big\|f(t,x,y)\big\|_{L^1_{x}}\right\|_{L^2_{y}}\\
&\lesssim
\big\|\widehat{K}(t,\xi)\big\|_{L^2_\xi}\left\|\big\|f(t,x,y)\big\|_{L^2_{y}}\right\|_{L^1_{x}}\\
&\lesssim
\big\|\widehat{K}(t,\xi)\big\|_{L^2_\xi}\big\||\nabla|^{\frac12-\epsilon}\langle\nabla\rangle^{2\epsilon} f\big\|_{L^1_{xy}}.
\end{align*}
This proves the lemma.
\end{proof}

\begin{lemma}\label{lem:tool4}
Assume that $\big\|\widehat{K}(t,\cdot)\big\|_{L^\infty}$ is bounded. Then, for any space-time Schwartz function $f$, and any $\epsilon>0$,
\begin{align}\label{tool4}
\big\|K(t,\partial_x)f\big\|_{L^2_{xy}}\lesssim \big\|\widehat{K}(t,\xi)\big\|_{L^\infty_\xi}\big\||\nabla|^{\frac12-\epsilon}\langle\nabla\rangle^{\frac12+2\epsilon} f\big\|_{L^1_{xy}}.
\end{align}
\end{lemma}
\begin{proof}
By a similar manner, we have
\begin{align*}
\big\|K(t,\partial_x)f\big\|_{L^2_{xy}}
&=\Big\|\big\|\widehat{K}(t,\xi)\mathcal F_\xi f(t,\xi,y)\big\|_{L^2_{\xi}}\Big\|_{L^2_{y}}\\
&\lesssim
\big\|\widehat{K}(t,\xi)\big\|_{L^\infty_\xi}\left\|\big\|\mathcal F_\xi f(t,\xi,y)\big\|_{L^2_{\xi}}\right\|_{L^2_{y}}\\
&=
\big\|\widehat{K}(t,\xi)\big\|_{L^\infty_\xi}\left\|\big\|f(t,x,y)\big\|_{L^2_{x}}\right\|_{L^2_{y}}\\
&\lesssim
\big\|\widehat{K}(t,\xi)\big\|_{L^\infty_\xi}\big\||\nabla|^{\frac12-\epsilon}\langle\nabla\rangle^{\frac12+2\epsilon} f\big\|_{L^1_{xy}}.
\end{align*}
This proves the lemma.
\end{proof}

Combining Lemmas \ref{lem:tool3} and \ref{lem:tool4}, we have
\begin{cor}\label{cor:too1}
Assume the Fourier multiplier operator $K(t,\partial_x)$ satisfies
$$
\big\|\widehat{\partial_x^\alpha K}(t,\xi)\big\|_{L^2_\xi(|\xi|\le\frac12)}<\infty,
\quad \big\|\widehat{K}(t,\xi)\big\|_{L^\infty_\xi(|\xi|\ge\frac12)} <\infty, \quad \alpha\ge 0.
$$
Then, for any space-time Schwartz function $f$ and any $\epsilon>0$,
\begin{align}
\big\|\partial_x^\alpha K(t,\partial_x)f\big\|_{L^2_{xy}}\lesssim&
\Big(\big\|\widehat{\partial_x^\alpha K}(t,\xi)\big\|_{L^2_\xi(|\xi|\le\frac12)}
+\big\|\widehat{K}(t,\xi)\big\|_{L^\infty_\xi(|\xi|\ge\frac12)}\Big) \nonumber\\
&\cdot\big\|\langle\nabla\rangle^{\alpha+\frac12+2\epsilon}|\nabla|^{\frac12-\epsilon} f\big\|_{L^1_{xy}}.
\end{align}
\end{cor}

\vskip .3in
\section{Proof of Lemma \ref{yl1}}
\label{sec3}

This section provides the proof of Lemma \ref{yl1}. More precisely, we show that
\begin{align*}
\|(\vec u,\, \psi)\|_{Y}\lesssim \|(\vec u,\, \psi)\|_{X} +  Q(\|(\vec u,\, \psi)\|_X),
\end{align*}
where $X,Y$ and $Q$ are defined as in Introduction.

\begin{proof}[Proof of Lemma \ref{yl1}]
First, we recall the basic inequality
\begin{equation}\label{e2.3}
\begin{aligned}
\|g(x,\,y)\|_{L^\infty_{xy}}&\lesssim \|g\|^{\frac14}_{L^2}\, \|\partial_x g\|^{\frac14}_{L^2}\, \|\partial_y g\|^{\frac14}_{L^2}\,\|\partial_x\partial_y g\|^{\frac14}_{L^2}
\\
&\lesssim \|\la\na\ra\pp_x g\|^\frac12_{L^2_{xy}}\|\la\na\ra g\|_{L^2_{xy}}^\frac12.
\end{aligned}
\end{equation}
We now estimate each term in $Y$ and start with the terms related to $\psi$. By interpolation,
\begin{align}\label{e2.41}
\|\la\na\ra^3\pp_x\psi(s,\,x,\,y)\|_{L^2_{xy}}&\lesssim \|\la\na\ra^3\pp_{xx}\psi\|^\frac12_{L^2_{xy}}
\|\la\na\ra^3\psi\|_{L^2_{xy}}^\frac12 \nonumber\\
&\lesssim s^{-\frac58} \|(\vec u,\,\psi)\|^\frac12_{X}\,\,s^{-\frac18} \|(\vec u,\,\psi)\|^\frac12_{X} \nonumber\\
&\lesssim s^{-\frac34} \|(\vec u,\,\psi)\|_{X}.
\end{align}
By \eqref{e2.3} and \eqref{e2.41}, we obtain
\begin{align}\label{e2.42}
\|\la\na\ra^2\pp_x\psi(s,\,x,\,y)\|_{L^\infty_{xy}}\lesssim \|\la\na\ra^3\pp_{xx}\psi\|^\frac12_{L^2_{xy}}
\|\la\na\ra^3\pp_x\psi\|_{L^2_{xy}}^\frac12\lesssim s^{-1}\|(\vec u,\,\psi)\|_{X}
\end{align}
and
\begin{align}\label{e2.43}
\|\la\na\ra^2\psi(s,\,x,\,y)\|_{L^\infty_{xy}}\lesssim \|\la\na\ra^3\pp_{x}\psi\|^\frac12_{L^2_{xy}}
\|\la\na\ra^3\psi\|_{L^2_{xy}}^\frac12\lesssim s^{-\frac12}\|(\vec u,\,\psi)\|_{X}.
\end{align}

\vskip .1in
The other terms in $Y$ are a little tricky. We first construct the following two inequalities, for $j=1,\,2$
and some $0<\a<1$,
\begin{align}\label{e2.5}
\|\la\na\ra \Pi_j\|_2\lesssim s^{-\frac54}\|(\vec u,\,\psi)\|_{Y}^\a\,\|(\vec u,\,\psi)\|_X^{2-\a};
\\\label{e2.6}
\|\Pi_j\|_\infty\lesssim s^{-\frac32}\|(\vec u,\,\psi)\|_{Y}^\a\,\|(\vec u,\,\psi)\|_X^{2-\a}.
\end{align}
Since the cases $j=1$ and $j=2$ can be treated the same way, we only deal with the case $j=1$.
According to the expression  of  $\Pi_1$ in \eqref{N1},
\begin{align*}
\|\la\na\ra \Pi_1\|_2\lesssim\,&\|\la\na\ra v\|_2\,\|\la\na\ra\pp_x u\|_\infty
+\|\la\na\ra v\|_\infty\,\|\la\na\ra^2u\|_2+\|\la\na\ra^3 \psi\|_2\,\|\la\na\ra^2\pp_x \psi\|_\infty
\\
 &+\|\la\na\ra u\|_2\,\|\la\na\ra\pp_x u\|_\infty+\|\la\na\ra u\|_2\,\|\la\na\ra\pp_x v\|_\infty
 +\|\la\na\ra \pp_{xx}\psi\|_\infty\,\|\la\na\ra^2 \psi\|_2
\\
=&I_1+I_2+I_3+{\rm remainder\, terms}.
\end{align*}
The first three terms $I_1$, $I_2$ and $I_3$ are typical of the terms on the right, and the estimates
of remainder terms are similar to them. Therefore, we shall only present their estimates.
\begin{align*}
I_1=&\|\la\na\ra v\|_2\,\|\la\na\ra\pp_x u\|_\infty
\\
\lesssim\,&\|\la\na\ra^N v\|^\frac 1N_2 \| v\|^{1-\frac 1N}_2\,\|\la\na\ra\pp_x u\|_\infty
\\
\lesssim &s^{\frac\va N}\|(\vec u,\psi)\|_X^\frac1N\cdot s^{-\frac54(1-\frac1N)}\| (\vec u,\psi)\|_Y^{1-\frac1N}\cdot s^{-1}\|(\vec u,\psi)\|_X\\
\lesssim &s^{-\frac94+\frac54\cdot\frac1N+\frac\va N}\|(\vec u,\psi)\|_Y^{1-\frac1N}\cdot\|(\vec u,\psi)\|_X^{1+\frac1N}\\
\lesssim & s^{-\frac54}\|(\vec u,\psi)\|_Y^{1-\frac1N}\cdot\|(\vec u,\psi)\|_X^{1+\frac1N},
\end{align*}
where $\va>0$ is a small parameter and $N$ has been chosen big enough.
For the second term, by interpolation we obtain
\begin{align*}
I_2=&\|\la\na\ra v\|_\infty\,\|\la\na\ra^2u\|_2
\\
\lesssim &\|\la\na\ra^N v\|_2^{1+\frac{1-\va}{N-1}}\|v\|^{1-\frac{1-\va}{N-1}}_\infty
\\
\lesssim &s^{\frac{1-\va}{N-1}\cdot\va}\|(\vec u,\,\psi)\|_X^{\frac{1-\va}{N-1}}\, s^{-\frac32(1-\frac{1-\va}{N-1})}\|(\vec u,\,\psi)\|_Y^{1-\frac{1-\va}{N-1}}\,s^\va\|(\vec u,\,\psi)\|_X\\
\lesssim & s^{-\frac54}\|(\vec u,\,\psi)\|_Y^{1-\frac{1-\va}{N-1}}\cdot\|(\vec u,\,\psi)\|_X^{1+\frac{1-\va}{N-1}}.
\end{align*}
By \eqref{e2.42}, we have
\begin{align*}
I_3=\|\la\na\ra^3 \psi\|_2\,\|\la\na\ra^2\pp_x \psi\|_\infty\lesssim s^{-\frac14}\|(\vec u,\,\psi)\|_X\cdot s^{-1}\|(\vec u,\,\psi)\|_X\lesssim s^{-\frac54}\|(\vec u,\,\psi)\|_X^2.
\end{align*}
Thus, we have obtained that, for some $\a_1\in (0,\,1)$,
\begin{align*}
\|\la\na\ra \Pi_1\|_2\lesssim s^{-\frac54}\|(\vec u,\,\psi)\|_X^{2-\a_1}\|(\vec u,\,\psi)\|_Y^{\a_1}.
\end{align*}
Through similar computations, for some $\a_2\in (0,\,1)$, we can get
\begin{align*}
\|\la\na\ra \pp_x\Pi_1\|_2\lesssim s^{-\frac74}\|(\vec u,\,\psi)\|_X^{2-\a_2}\|(\vec u,\,\psi)\|_Y^{\a_2}.
\end{align*}
Therefore, by Nash's inequality, we obtain, for some $\a\in (0,\,1)$,
\begin{align*}
\| \Pi_1\|_\infty\lesssim \|\la\na\ra\pp_x\Pi_1\|^\frac12_2\,\|\la\na\ra \Pi_1\|^\frac12_2\lesssim s^{-\frac32}\|(\vec u,\,\psi)\|_X^{2-\a}\|(\vec u,\,\psi)\|_Y^\a.
\end{align*}
Thus we complete the proof of  \eqref{e2.5} and \eqref{e2.6}.

\vskip .1in
Now, using the inequalities \eqref{e2.41}, \eqref{e2.42},  \eqref{e2.5},  \eqref{e2.6} with $j=1$ and \eqref{u1}, we have for small $\eta_0>0$, $s\ge1$,
\begin{align*}
\|u\|_\infty\lesssim\,& \|\pp_tu\|_\infty +\|\la\na\ra\pp_x\psi\|_\infty+\|\Pi_1\|_\infty
\\
\lesssim\,& s^{-\frac32}\|(\vec u,\,\psi)\|_X+s^{-1}\|(\vec u,\,\psi)\|_X+s^{-\frac32}\|(\vec u,\,\psi)\|_Y^{\a}\|(\vec u,\,\psi)\|_X^{2-\a}
\\
\lesssim\,& s^{-1}(\|(\vec u,\,\psi)\|_X+C(\eta_0)\,\|(\vec u,\,\psi)\|_X^2)+s^{-1}\eta_0\|(\vec u,\,\psi)\|_Y.
\end{align*}
\begin{align*}
\|u\|_2\lesssim\,& \|\pp_tu\|_2+\|\la\na\ra\pp_x\psi\|_2+\|\Pi_1\|_2
\\
\lesssim& s^{-\frac54}\|(\vec u,\,\psi)\|_X+s^{-\frac34}\|(\vec u,\,\psi)\|_X+s^{-\frac54}\|(\vec u,\,\psi)\|_Y^{\a}\|(\vec u,\,\psi)\|_X^{2-\a}
\\
\lesssim& s^{-\frac34}(\|(\vec u,\,\psi)\|_X+C(\eta_0)\,\|(\vec u,\,\psi)\|_X^2)+s^{-\frac34}\eta_0\|(\vec u,\,\psi)\|_Y.
\end{align*}
\begin{align*}
\|\pp_x u\|_2\lesssim\,& \|\pp_x\pp_t u\|_2 +\|\la\na\ra\pp_{xx}\psi\|_2+\|\pp_x\Pi_1\|_2
\\
\lesssim&\|\la\na\ra \pp_tu\|_2 +\|\la\na\ra\pp_{xx}\psi\|_2+\|\la\na\ra \Pi_1\|_2
\\
\lesssim&\|\la\na\ra \pp_tu\|_2 +\|\la\na\ra\pp_{xx}\psi\|_2+
s^{-\frac54}\|(\vec u,\,\psi)\|_Y^{\a}\|(\vec u,\,\psi)\|_X^{2-\a}
\\
\lesssim& s^{-\frac54}(\|(\vec u,\,\psi)\|_X+C(\eta_0)\,\|(\vec u,\,\psi)\|_X^2)+s^{-\frac54}\eta_0\|(\vec u,\,\psi)\|_Y.
\end{align*}
Similarly, by equation \eqref{v1} and inequalities \eqref{e2.42}, \eqref{e2.5} and \eqref{e2.6} with $j=2$,
we have
\begin{align*}
\|v\|_\infty\lesssim \, s^{-\frac32}(\|(\vec u,\,\psi)\|_X+C(\eta_0)\,\|(\vec u,\,\psi)\|_X^2)+s^{-\frac32}\eta_0\|(\vec u,\,\psi)\|_Y.
\\
\|v\|_2\lesssim \, s^{-\frac54}(\|(\vec u,\,\psi)\|_X+C(\eta_0)\,\|(\vec u,\,\psi)\|_X^2)+s^{-\frac54}\eta_0\|(\vec u,\,\psi)\|_Y.
\end{align*}
Therefore, collecting the estimates above, we prove that
\begin{align*}
\|(\vec u,\,\psi)\|_Y\lesssim \, \|(\vec u,\,\psi)\|_X+\eta_0\|(\vec u,\,\psi)\|_Y + \|(\vec u,\,\psi)\|_X^2.
\end{align*}
Choosing $\eta_0$ small enough, we get
\begin{align*}
\|(\vec u,\,\psi)\|_Y\lesssim \, \|(\vec u,\,\psi)\|_X + \|(\vec u,\,\psi)\|_X^2.
\end{align*}
This completes the proof of Lemma \ref{yl1}.
\end{proof}

\vskip .3in
\section{Energy estimates}
\label{sec:energy}

The rest of the sections are devoted to proving \eqref{con1}, namely
$$
\|(\vec u,\, \psi)\|_{X}\lesssim \|(\vec u_0,\, \psi_0)\|_{X_0} + Q\big(\|(\vec u,\, \psi)\|_Y\big).
$$
This section shows that the first term in the definition of the norm of $X$ obeys this inequality.

\vskip .1in
For this purpose, we first show that, for any real number $\sigma>0$,
\begin{align}
& \frac{d}{dt}\Big(\big\|\la\na\ra^\sigma \vec u\big\|_{L^2}^2 +\big\|\la\na\ra^\sigma \vec b\big\|_{L^2}^2\Big)\nonumber \\
\lesssim & \Big(\|\nabla\vec u\|_{L^\infty}+\|\nabla\vec b\|^2_{L^\infty}\Big)
\Big(\|\la\na\ra^\sigma\vec u\|_{L^2}^2   +  \|\la\na\ra^\sigma\vec b\|_{L^2}^2\Big).\label{hsnorm}
\end{align}
To do so, we take advantage of (\ref{e1.11}), which is equivalent to \eqref{e1.1}. Applying
$\la\na\ra^\sigma$ to (\ref{e1.11}) and taking the inner product with
$(\la\na\ra^\sigma \vec u, \la\na\ra^\sigma \vec b)$, we obtain, after integrating by parts and
invoking $\na\cdot \vec u=0$,
\begin{align}
& \frac12 \frac{d}{dt}\Big(\big\|\la\na\ra^\sigma \vec u\big\|_{L^2}^2
+ \big\|\la\na\ra^\sigma \vec b\big\|_{L^2}^2\Big) + \big\|\la\na\ra^\sigma \vec u\big\|_{L^2}^2 \nonumber\\
&\qquad = -\int [\la\na\ra^\sigma,\vec u\cdot\nabla] \vec u \cdot \la\na\ra^\sigma \vec u\,dx
- \int [\la\na\ra^\sigma,\vec u\cdot\nabla] \vec b \cdot \la\na\ra^\sigma \vec b\,dx \nonumber\\
&\qquad + \int [\la\na\ra^\sigma, \vec b\cdot\nabla] \vec b \cdot \la\na\ra^\sigma \vec u\,dx \label{mides}
+ \int [\la\na\ra^\sigma,\vec b\cdot\nabla] \vec u \cdot \la\na\ra^\sigma \vec b\,dx,
\end{align}
where we have used the standard commutator notation
$$
[\la\na\ra^\sigma,\vec u\cdot\nabla] \vec u = \la\na\ra^\sigma (\vec u\cdot\nabla \vec u)
-\vec u\cdot\nabla \la\na\ra^\sigma \vec u.
$$
By H\"{o}lder's inequality and a standard commutator estimate, we have
\begin{align*}
\left|\int [\la\na\ra^\sigma,\vec u\cdot\nabla] \vec u \cdot \la\na\ra^\sigma \vec u\,dx\right| &\le
\big\|[\la\na\ra^\sigma,\vec u\cdot\nabla] \vec u \big\|_{L^2}\, \big\|\la\na\ra^\sigma \vec u\big\|_{L^2} \\
&\lesssim \big\|\la\na\ra^\sigma \vec u\big\|_{L^2}\, \|\nabla\vec u\|_{L^\infty} \, \big\|\la\na\ra^\sigma \vec u\big\|_{L^2},
\end{align*}
moreover, for some constant $C>0$,
\begin{align*}
&\left|\int [\la\na\ra^\sigma,\vec b\cdot\nabla] \vec u \cdot \la\na\ra^\sigma \vec b\,dx\right| \le
\big\|[\la\na\ra^\sigma,\vec b\cdot\nabla] \vec u \big\|_{L^2}\, \big\|\la\na\ra^\sigma \vec b\big\|_{L^2} \\
&\qquad \qquad \lesssim \Big(\big\|\la\na\ra^\sigma \vec u\big\|_{L^2}\, \|\nabla\vec b\|_{L^\infty} +\big\|\la\na\ra^\sigma \vec b\big\|_{L^2}\, \|\nabla\vec u\|_{L^\infty}\Big)  \, \big\|\la\na\ra^\sigma \vec b\big\|_{L^2}\\
&\qquad \qquad \le  \frac14 \|\la\na\ra^\sigma \vec u\big\|_{L^2}^2
+ C\Big(\|\nabla\vec b\|^2_{L^\infty} + \|\nabla\vec u\|_{L^\infty}\Big)\,
\big\|\la\na\ra^\sigma \vec b\big\|^2_{L^2}.
\end{align*}
The other two terms can be similarly bounded. We obtain \eqref{hsnorm} after we insert
the estimates above in (\ref{mides}). Setting $\sigma=N$ and integrating \eqref{hsnorm} in time,
we obtain
\begin{align*}
\big\|\la\na\ra^N (\vec u, \vec b )\big\|_{L^2}^2
\lesssim \|(\vec u_0,\,\psi_0)\|^2_{X_0} + \int_1^t
\Big(\|\nabla\vec u\|_{L^\infty}+\|\nabla\vec b\|^2_{L^\infty}\Big)
\cdot\big\|\la\na\ra^N (\vec u, \vec b ) \big\|_{L^2}^2\,d s.
\end{align*}
To bound $\|\nabla\vec b\|_{L^\infty}$, by the definition of $Y$,
\begin{align*}
\|\nabla\vec b(s)\|^2_{L^\infty} \lesssim \|\la\na\ra^2\psi(s)\|_{L^\infty}^2\lesssim s^{-1}\, \|(\vec u,\,\psi)\|^2_{Y}.
\end{align*}
Also, we have
$$
\|\nabla\vec u(s)\|_{L^\infty}\le \|\nabla u(s)\|_{L^\infty}+ \|\nabla v(s)\|_{L^\infty} \le s^{-1} \|(\vec u,\,\psi)\|_{Y}.
$$
Therefore,
\begin{align*}
\big\|\la\na\ra^N (\vec u, \vec b )\big\|_{L^2}^2
&\lesssim \|(\vec u_0,\,\psi_0)\|^2_{X_0}+\int_1^t\langle s\rangle^{-1+2\varepsilon}\,ds\>(\|(\vec u,\,\psi)\|_{Y}+\|(\vec u,\,\psi)\|^2_{Y})
\>\|(\vec u,\,\psi)\|^2_{Y}
\\
&\lesssim
\|(\vec u_0,\,\psi_0)\|^2_{X_0}+
 t^{2\varepsilon}\>(\|(\vec u,\,\psi)\|_{Y}^3+\|(\vec u,\,\psi)\|^4_{Y}).
\end{align*}
\vskip .2in
This proves that
$$
\sup_{t\geq 1}\left(t^{-\varepsilon}\big\|\la\na\ra^N\big(\vec u(t),\,\nabla \psi(t)\big)\big\|_2\right)
\lesssim
\|(\vec u_0,\,\psi_0)\|_{X_0}+
(\|(\vec u,\,\psi)\|_{Y}^3+\|(\vec u,\,\psi)\|^4_{Y})^{\frac12}.
$$

\vskip .3in
\section{Estimates on nonlinearities}
\label{sec:non}

This section estimates the nonlinear terms $F_0, F_1, F_2$,
defined in (\ref{e1.6}). These bounds will be used in the proof of (\ref{con1}) given in
the subsequent sections.

\vskip .1in
We will use the Littlewood-Paley projection operators. Let $\phi(\xi)$ be a
smooth bump function supported in the ball $|\xi| \le 2$ and equals one on the ball $|\xi|\le 1$.
For any real number $M>0$ and $f\in \mathcal{S}'$ (tempered distributions), the projection
operators can be defined as follows:
\begin{align*}
&\widehat{P_{\le M}f}(\xi) := \phi(\xi/M) \,\hat{f}(\xi),\\
&\widehat{P_{> M}f}(\xi) := \left(1-\phi(\xi/M)\right) \,\hat{f}(\xi),\\
&\widehat{P_{M}f}(\xi) := \left(\phi(\xi/M) - \phi(2\xi/M)\right) \,\hat{f}(\xi).
\end{align*}

We also need the following estimate involving the Riesz transform $\mathcal R$.
\begin{lemma}\label{lem:Riesz-bound}
For any $\epsilon >0$,
\begin{align}\label{Riesz-bound}
\big\|\mathcal R  f\big\|_{L^1_{xy}}\lesssim \big\||\nabla|^{-\epsilon}\langle\nabla\rangle^{2\epsilon}f\big\|_{L^1_{xy}}.
\end{align}
\end{lemma}
The proof of Lemma \ref{lem:Riesz-bound} is presented in Appendix \ref{app2}.

\vskip .1in
\begin{lemma}\label{lem:claim1}
For any $s\ge 1$,
\begin{align}\label{claim1}
\big\|\langle \nabla\rangle^5|\nabla|^{\frac12-\epsilon}F_1(s, \cdot)\big\|_{L^1_{xy}}
\lesssim
s^{-\frac32-\varepsilon}Q\big(\|(\vec u,\, \psi)\|_Y\big) ,
\end{align}
where $\epsilon$ is same as in Corollary \ref{cor:too1}, and $\varepsilon$ is same as in \eqref{X}.
\end{lemma}

\begin{proof}
By definition \eqref{Def:F1},
\begin{align}\label{clm1:1}
F_1=&-\partial_x\partial_y\big(u\>\partial_x\psi\big)
-\partial_x\partial_y\big(v\>\partial_y\psi\big)+\partial_t \Pi_1,
\end{align}
where $\Pi_1$ is explicitly given in Appendix \ref{app1}. Since $F_1$ is a quadratic nonlinearity, we write
$$
F_1=F_{11}(u,v)+F_{12}(u,\psi)+F_{13}(v,\psi)+F_{14}(u,u)+F_{15}(v,v)+F_{16}(\psi,\psi),
$$
where $F_{11}(u,v)$ is a collection of the terms which contain the unknown function $(u,v)$, and $F_{12},\cdots, F_{16}$
are defined similarly. To bound these terms, we split $F_1$ into low frequency parts and high frequency parts. More precisely,
for small $\delta>0$ to be specified later,
we write
\begin{align}
F_1:= F_{1,high} + F_{1,low},   \label{F1highlow}
\end{align}
where
\begin{align*}
&F_{1,high}=F_{11}\big((1-P_{\le s^\delta})u,v\big)+F_{11}\big(P_{\le s^\delta}u,(1-P_{\le s^\delta})v\big)+\cdots,\\
&F_{1,low}=F_{11}\big(P_{\le s^\delta} u, P_{\le s^\delta} v\big) +\cdots,
\end{align*}
that is, each term in $F_{1,high}$ contains at least one high frequency part and the terms in $F_{1,low}$ involve only
low frequencies.

\vskip .1in
Although the number of the terms in $F_1$ is large, they can be treated similarly. For the sake of clarity,
we shall only present the estimates for a representative term. That is, we write
\begin{align}
F_1=\frac{\partial_y\partial_y}{\Delta}\big(\partial_y\psi\>\partial_x\partial_y\psi_t\big)+\mbox{similar terms}.
\end{align}
We now focus on the representative term $\frac{\partial_y\partial_y}{\Delta}\big(\partial_y\psi\>\partial_x\partial_y\psi_t\big)$.
As in (\ref{F1highlow}), we split it into high and low frequency terms. First, we deal with those involving high frequencies,
which can be treated by a standard way (see \cite{LiWu} for some related analysis). We focus on $\frac{\partial_x\partial_y}{\Delta}\big(\partial_y\psi\>\partial_x\partial_yP_{\gtrsim s^\delta}\psi_t\big)$ and
by \eqref{Riesz-bound},
\begin{align*}
&\Big\|\langle\nabla \rangle^5|\nabla|^{\frac12-\epsilon}  \frac{\partial_y\partial_y}{\Delta}\big(\partial_y\psi\>\partial_x\partial_yP_{\gtrsim s^\delta}\psi_t\big)\Big\|_{L^1_{xy}}\\
\lesssim &
\Big\|\langle \nabla\rangle^{5+2\epsilon}|\nabla|^{\frac12-2\epsilon}  \big(\partial_y\psi\>\partial_x\partial_yP_{\gtrsim s^\delta}\psi_t\big)\Big\|_{L^1_{xy}}\\
\lesssim &
\Big\|\langle \nabla\rangle^{6}\big(\partial_y\psi\>\partial_x\partial_yP_{\gtrsim s^\delta}\psi_t\big)\Big\|_{L^1_{xy}}\\
\lesssim &
\Big\|\langle \nabla\rangle^{6}\partial_y\psi\Big\|_{L^2_{xy}}\>\big\|\partial_x\partial_yP_{\gtrsim s^\delta}\psi_t\big\|_{L^2_{xy}}
+ \big\|\partial_y\psi\big\|_{L^2_{xy}}\>\Big\|\langle \nabla\rangle^{6}\partial_x\partial_yP_{\gtrsim s^\delta}\psi_t\Big\|_{L^2_{xy}}\\
\lesssim &
\Big\|\langle \nabla\rangle^{7}\psi\Big\|_{L^2_{xy}}\>\Big\|\langle \nabla\rangle^{8}P_{\gtrsim s^\delta}\psi_t\Big\|_{L^2_{xy}}.
\end{align*}
Since $\psi_t=-v-\vec u\cdot \nabla \psi$, we have, for the enough large $N$,
\begin{align*}
\Big\|\langle \nabla\rangle^{8}P_{\gtrsim s^\delta}\psi_t\Big\|_{L^2_{xy}}
\lesssim &
\Big\|\langle \nabla\rangle^{8}P_{\gtrsim s^\delta}v\Big\|_{L^2_{xy}}+
\Big\|\langle \nabla\rangle^{8}P_{\gtrsim s^\delta}\big(\vec u\cdot \nabla \psi\big)\Big\|_{L^2_{xy}}\\
\lesssim &
s^{-(N-8)\delta}\Big(\big\|\langle \nabla\rangle^{N}v\big\|_{L^2_{xy}}+
\big\|\langle \nabla\rangle^{N}\big(\vec u\cdot \nabla \psi\big)\big\|_{L^2_{xy}}\Big)\\
\lesssim &
s^{-(N-8)\delta}\Big(\big\|\langle \nabla\rangle^{N}v\big\|_{L^2_{xy}}+
\big\|\langle \nabla\rangle^{N}\vec u\big\|_{L^2_{xy}}\cdot \big\|\langle \nabla\rangle^{N+1} \psi\big\|_{L^2_{xy}}\Big)\\
\lesssim &
s^{-\frac32-2\varepsilon}\Big(\|(\vec u, \psi)\|_Y+\|(\vec u, \psi)\|_Y^2\Big).
\end{align*}
 Therefore, combining with above two estimates, we obtain that
$$
\Big\|\langle\nabla \rangle^{5}|\nabla|^{\frac12-\epsilon}  \frac{\partial_y\partial_y}{\Delta}\big(\partial_y\psi\>\partial_x\partial_yP_{\gtrsim s^\delta}\psi_t\big)\Big\|_{L^1_{xy}}
\lesssim
s^{-\frac32-\varepsilon}Q\big(\|(\vec u,\, \psi)\|_Y\big) .
$$

Now we turn to  $F_{1,low}$. Again, using \eqref{Riesz-bound}, we have
\begin{align*}
&\Big\|\langle\nabla \rangle^{5}|\nabla|^{\frac12-\epsilon}  \frac{\partial_y\partial_y}{\Delta}\big(\partial_yP_{\le s^\delta}\psi\>\partial_x\partial_yP_{\le s^\delta}\psi_t\big)\Big\|_{L^1_{xy}}\\
=&\Big\|\langle\nabla \rangle^{5}|\nabla|^{\frac12-\epsilon}  \frac{\partial_y\partial_y}{\Delta}P_{\le 4s^\delta}\big(\partial_yP_{\le s^\delta}\psi\>\partial_x\partial_yP_{\le s^\delta}\psi_t\big)\Big\|_{L^1_{xy}}\\
\lesssim &
\Big\|\langle\nabla \rangle^{5+2\epsilon}|\nabla|^{\frac12-2\epsilon} P_{\le 4s^\delta} \big(\partial_yP_{\le s^\delta}\psi\>\partial_x\partial_yP_{\le s^\delta}\psi_t\big)\Big\|_{L^1_{xy}}\\
\lesssim &
s^{6\delta} \big\|\partial_yP_{\le s^\delta}\psi\>\partial_x\partial_yP_{\le s^\delta}\psi_t\big\|_{L^1_{xy}}\\
\lesssim &
s^{6\delta} \big\|\partial_yP_{\le s^\delta}\psi\big\|_{L^2_{xy}}\big\|\partial_x\partial_yP_{\le s^\delta}\psi_t\big\|_{L^2_{xy}}\\
\lesssim &
s^{8\delta} \big\|\psi\big\|_{L^2_{xy}}\big\|\partial_x\psi_t\big\|_{L^2_{xy}}.
\end{align*}
Now we need the following estimate,
\begin{align}\label{e2.7}
\big\|\partial_x\psi_t\big\|_{L^2_{xy}}\lesssim s^{-\frac32}\big(\|(\vec{u},\psi)\|_Y+Q\big(\|(\vec u,\, \psi)\|_Y\big)\big).
\end{align}
Indeed, by equations \eqref{v1} and \eqref{p1}, we have
\begin{align}
\big\|\partial_x\psi_t\big\|_{L^2_{xy}}
\lesssim &
\|\partial_x v\|_2+\|\partial_x (\vec u\cdot\nabla \psi)\|_{2}
\notag\\
\lesssim &
\|\partial_x\partial_t v\|_2+\|\partial_{x}^3\psi\|_2+\big(\|\partial_x u\|_2+\|\partial_x v\|_2\big)\|\langle \nabla\rangle \psi\|_\infty\notag\\
&+\big(\| u\|_\infty+\|v\|_\infty\big)\|\langle \nabla\rangle \partial_x\psi\|_2 +\|\partial_x \Pi_2\|_2. \label{21.17}
\end{align}
As in the proof of \eqref{e2.6}, we have
$$
\|\partial_x \Pi_2\|_2\lesssim s^{-\frac32}Q\big(\|(\vec u,\, \psi)\|_Y\big).
$$
Then, \eqref{e2.7} follows from \eqref{21.17}. By \eqref{e2.7}, for any $8\delta\le \frac14+\varepsilon$, we get
\begin{align*}
\Big\|\langle\nabla \rangle^{5}|\nabla|^{\frac12-\epsilon} \frac{\partial_y\partial_y}{-\Delta}  \big(\partial_yP_{\le s^\delta}\psi\>\partial_x\partial_yP_{\le s^\delta}\psi_t\big)\Big\|_{L^1_{xy}}
\lesssim&
s^{8\delta} s^{-\frac74}\|(\vec u,\psi)\|_Y^2\\
\lesssim &
s^{-\frac32-\varepsilon}Q\big(\|(\vec u,\, \psi)\|_Y\big).
\end{align*}
This proves the lemma.
\end{proof}

\vskip .1in
\begin{lemma} \label{lem:claim2}
Let $\epsilon,\varepsilon$ be the same as in Lemma \ref{claim1}. Then for any $s\ge 1$,
\begin{align}\label{claim2}
\big\|\langle \nabla\rangle^5|\nabla|^{\frac12-\epsilon}F_2\big\|_{L^1_{xy}}
\lesssim
s^{-\frac32-\varepsilon}Q\big(\|(\vec u,\, \psi)\|_Y\big).
\end{align}
\end{lemma}
\begin{proof}
As before, we write
\begin{align}\label{clm2:1}
F_2=&\partial_t\Pi_2+\partial_x\partial_x\big(\vec u\>\nabla\psi\big)\notag\\
=&-2\frac{\partial_x\partial_y}{\Delta}\big(\partial_y\psi\>\partial_x\partial_y\psi_t\big)+\mbox{similar terms}.
\end{align}
Since \eqref{clm2:1} has the similar form as \eqref{clm1:1}, we have the same estimate as the one for $F_1$. The details are omitted here.
\end{proof}

The nonlinear term $F_0$ behaves much differently from $F_1,F_2$, indeed, we have
\begin{lemma} \label{F0bd}
Let $\epsilon,\varepsilon$ be the same as in Lemma \ref{claim1}. Then for any $s\ge 1$,
\begin{align}
\big\|\langle \nabla\rangle^{5}|\nabla|^{\frac12-\epsilon}F_0\big\|_{L^1_{xy}}
\lesssim &
s^{-1-\varepsilon}Q\big(\|(\vec u,\, \psi)\|_Y\big) ;\label{claim31}\\
\big\|\langle \nabla\rangle^{\frac{11}2+2\epsilon}\nabla F_0\big\|_{L^1_{xy}}
\lesssim &
s^{-\frac32+\varepsilon}Q\big(\|(\vec u,\, \psi)\|_Y\big) ;\label{claim32}\\
\big\|\langle \nabla\rangle^{5}\nabla \partial_x F_0\big\|_{L^1_{xy}}
\lesssim &
s^{-\frac32-\varepsilon}Q\big(\|(\vec u,\, \psi)\|_Y\big).\label{claim33}
\end{align}
\end{lemma}
\begin{proof}
By \eqref{Def:F0} we write
\begin{align}\label{clm3:1}
F_0=&-(\vec u+\vec u_t)\cdot \nabla\phi-\Pi_2\notag\\
=&- \Pi_2-\partial_x\partial_y \psi\>\partial_x\psi-\Pi_1\,\partial_x \psi+\partial_x^2\psi\>\partial_y\psi-\Pi_2\>\partial_y \psi-u\>\partial_x \psi_t-v\>\partial_y \psi_t\notag\\
=&-\frac{\partial_x\partial_y}{\Delta}\big(\partial_y\psi\>\partial_x\partial_y\psi\big)+\mbox{similar terms}.
\end{align}
We also write
$$
F_0:=F_{0,low}+F_{0,high},
$$
where
$$
F_{0,low}=F_0\big(P_{\le s^\delta}u, P_{\le s^\delta}v,P_{\le s^\delta}\psi\big).
$$
We only consider $F_{0,low}$, since the piece $F_{0,high}$ can be treated in the manner as in the proof of Lemma \ref{lem:claim1}. We need the following lemma.

\begin{lemma}\label{lem:RR}
Let $\alpha>0, N_0\ge 0$, then for any $\epsilon>0$,
\begin{align}
\Big\||\nabla|^\alpha\langle\nabla \rangle^{N_0} \frac{\partial_x\partial_y}{-\Delta}f\Big\|_{L^1_{xy}}\lesssim
\|f\|_{L^1_{xy}}^{1-\alpha+\epsilon}\|\partial_x f\|_{L^1_{xy}}^{\alpha-\epsilon}+\big\|\langle\nabla\rangle^{N_0-1+\alpha+\epsilon}\partial_x f\big\|_{L^1_{xy}}.
\end{align}
\end{lemma}
The proof of Lemma \ref{lem:RR} is presented in Appendix \ref{app3}.

\vskip .1in
To prove \eqref{claim31}, we use Lemma \ref{lem:RR} for $\alpha=\frac12-\epsilon$ to get
\begin{align*}
&\Big\|\langle\nabla \rangle^{\frac{11}2+2\epsilon}|\nabla|^{\frac12-\epsilon}\frac{\partial_x\partial_y}{-\Delta}\big(\partial_yP_{\le s^\delta}\psi\>\partial_x\partial_yP_{\le s^\delta}\psi\big)\Big\|_{L^1_{xy}}\\
\lesssim &
\Big\|\partial_yP_{\le s^\delta}\psi\>\partial_x\partial_yP_{\le s^\delta}\psi\Big\|_{L^1_{xy}}^{\frac12+2\epsilon}
\Big\|\partial_x\big(\partial_yP_{\le s^\delta}\psi\>\partial_x\partial_yP_{\le s^\delta}\psi\big)\Big\|_{L^1_{xy}}^{\frac12-2\epsilon}\\
&
+\Big\|\langle\nabla\rangle^{6+2\epsilon}\partial_x\big(\partial_yP_{\le s^\delta}\psi\>\partial_x\partial_yP_{\le s^\delta}\psi\big)\Big\|_{L^1_{xy}}.
\end{align*}
By H\"older's inequality, we have
\begin{align}\label{e5.15}
&\Big\|\partial_yP_{\le s^\delta}\psi\>\partial_x\partial_yP_{\le s^\delta}\psi\Big\|_{L^1_{xy}}
\lesssim
\big\|\partial_yP_{\le s^\delta}\psi\big\|_{L^2_{xy}}\big\|\partial_x\partial_yP_{\le s^\delta}\psi\big\|_{L^2_{xy}}\notag\\
\lesssim &
s^{2\delta}\big\|\psi\big\|_{L^2_{xy}}\big\|\partial_x\psi\big\|_{L^2_{xy}}
\lesssim
s^{2\delta-1}\big\|(\vec u,\psi)\big\|_Y^2.
\end{align}
Moreover,
\begin{align}\label{e5.16}
&\Big\|\partial_x\big(\partial_yP_{\le s^\delta}\psi\>\partial_x\partial_yP_{\le s^\delta}\psi\big)\Big\|_{L^1_{xy}}\notag\\
\lesssim &
\big\|\partial_x\partial_yP_{\le s^\delta}\psi\big\|_{L^2_{xy}}^2+
\big\|\partial_yP_{\le s^\delta}\psi\big\|_{L^2_{xy}}\big\|\partial_x^2\partial_yP_{\le s^\delta}\psi\big\|_{L^2_{xy}}\notag\\
\lesssim &
s^{2\delta}\big\|\partial_x\psi\big\|_{L^2_{xy}}^2+
s^{2\delta}\big\|\psi\big\|_{L^2_{xy}}\big\|\partial_x^2\psi\big\|_{L^2_{xy}}\notag\\
\lesssim &
s^{2\delta-\frac32}\big\|(\vec u,\psi)\big\|_Y^2
\end{align}
and using \eqref{e5.16},
\begin{align}\label{e5.17}
\big\|\langle\nabla\rangle^{6+2\epsilon}\partial_x\big(\partial_yP_{\le s^\delta}\psi\>\partial_x\partial_yP_{\le s^\delta}\psi\big)\big\|_{L^1_{xy}}
\lesssim&
s^{(6+2\epsilon)\delta}\Big\|\partial_x\big(\partial_yP_{\le s^\delta}\psi\>\partial_x\partial_yP_{\le s^\delta}\psi\big)\Big\|_{L^1_{xy}}\notag\\
\lesssim &
s^{9\delta-\frac32}\big\|(\vec u,\psi)\big\|_Y^2.
\end{align}
Therefore, combining \eqref{e5.15}--\eqref{e5.17},  we deduce that, for any  $\delta>0$ satisfying $9\delta +\varepsilon\le \frac12$,
\begin{align*}
&\Big\|\langle\nabla \rangle^{5}|\nabla|^{\frac12-\epsilon}\frac{\partial_x\partial_y}{-\Delta}\big(\partial_yP_{\le s^\delta}\psi\>\partial_x\partial_yP_{\le s^\delta}\psi\big)\Big\|_{L^1_{xy}}
\lesssim
s^{-1-\varepsilon}\big\|(\vec u,\psi)\big\|_Y^2.
\end{align*}
For \eqref{claim32}, we only need to prove
$$
\Big\|\langle\nabla \rangle^{5}\nabla \frac{\partial_x\partial_y}{-\Delta}\big(\partial_yP_{\le s^\delta}\psi\>\partial_x\partial_yP_{\le s^\delta}\psi\big)\Big\|_{L^1_{xy}}
\lesssim
s^{-\frac32+\varepsilon}Q\big(\|(\vec u,\, \psi)\|_Y\big).
$$
By Lemma \ref{lem:Riesz-bound} and Lemma \ref{lem:RR} for $\alpha=1-\epsilon$, and by \eqref{e5.15}--\eqref{e5.17} we have
\begin{align*}
&\Big\|\langle\nabla \rangle^{5}\nabla\frac{\partial_x\partial_y}{-\Delta}\big(\partial_yP_{\le s^\delta}\psi\>\partial_x\partial_yP_{\le s^\delta}\psi\big)\Big\|_{L^1_{xy}}\\
\lesssim &
\Big\|\langle\nabla \rangle^{5+2\epsilon}|\nabla|^{1-\epsilon}\frac{\partial_x\partial_y}{-\Delta} \big(\partial_yP_{\le s^\delta}\psi\>\partial_x\partial_yP_{\le s^\delta}\psi\big)\Big\|_{L^1_{xy}}\\
\lesssim &
\Big\|\partial_yP_{\le s^\delta}\psi\>\partial_x\partial_yP_{\le s^\delta}\psi
\Big\|_{L^1_{xy}}^{2\epsilon}
\Big\|\partial_x\big(\partial_yP_{\le s^\delta}\psi\>\partial_x\partial_yP_{\le s^\delta}\psi\big)\Big\|_{L^1_{xy}}^{1-2\epsilon}\\
&
+\Big\|\langle\nabla\rangle^{6+2\epsilon}\partial_x\big(\partial_yP_{\le s^\delta}\psi\>\partial_x\partial_yP_{\le s^\delta}\psi\big)\Big\|_{L^1_{xy}}\\
\lesssim &
\big(s^{2\epsilon(2\delta-1)+(1-2\epsilon)(2\delta-\frac32)}
+s^{9\delta-\frac32}\big)\big\|(\vec u,\psi)\big\|_Y^2\\
\lesssim &
s^{-\frac32+\varepsilon}\big\|(\vec u,\psi)\big\|_Y^2.
\end{align*}
This gives \eqref{claim32}. To prove \eqref{claim33}, we only need to prove
$$
\Big\|\langle\nabla \rangle^5\nabla\frac{\partial_x\partial_y}{-\Delta}\partial_x\big(\partial_yP_{\le s^\delta}\psi\>\partial_x\partial_yP_{\le s^\delta}\psi\big)\Big\|_{L^1_{xy}}
\lesssim
s^{-\frac32-\varepsilon}Q\big(\|(\vec u,\, \psi)\|_Y\big).
$$
Again,
\begin{align*}
&\Big\|\langle\nabla \rangle^5\nabla\frac{\partial_x\partial_y}{-\Delta}\partial_x\big(\partial_yP_{\le s^\delta}\psi\>\partial_x\partial_yP_{\le s^\delta}\psi\big)\Big\|_{L^1_{xy}}\\
\lesssim &
\Big\|\langle\nabla \rangle^{5+2\epsilon}|\nabla|^{1-\epsilon}\frac{\partial_x\partial_y}{-\Delta}\partial_x\big(\partial_yP_{\le s^\delta}\psi\>\partial_x\partial_yP_{\le s^\delta}\psi\big)\Big\|_{L^1_{xy}}\\
\lesssim &
\Big\|\partial_x\big(\partial_yP_{\le s^\delta}\psi\>\partial_x\partial_yP_{\le s^\delta}\psi\big)\Big\|_{L^1_{xy}}^{2\epsilon}
\Big\|\partial_x^2\big(\partial_yP_{\le s^\delta}\psi\>\partial_x\partial_yP_{\le s^\delta}\psi\big)\Big\|_{L^1_{xy}}^{1-2\epsilon}\\
&
+\Big\|\langle\nabla\rangle^{6+2\epsilon}\partial_x^2\big(\partial_yP_{\le s^\delta}\psi\>\partial_x\partial_yP_{\le s^\delta}\psi\big)\Big\|_{L^1_{xy}}\\
\lesssim &
\big(s^{2\epsilon(2\delta-\frac32)+(1-2\epsilon)(2\delta-\frac74)}+s^{9\delta-\frac74}\big)\big\|(\vec u,\psi)\big\|_Y^2\\
\lesssim &
s^{-\frac32-\varepsilon}\big\|(\vec u,\psi)\big\|_Y^2,
\end{align*}
where $\delta>0$ and satisfies $9\delta+\varepsilon\le \frac14$. This gives \eqref{claim33}. This finishes the proof of
Lemma \ref{F0bd}.
\end{proof}

\vskip .3in
\section{Estimates on $\|\la\na\ra\pp_{xx}\psi\|_\infty,\,
\|\pp_t \vec{u}\|_\infty,\,\|\la\na\ra\pp_x \vec{u}\|_\infty$}
\label{s2}

\vskip .1in
This section continues the proof for \eqref{con1}. For the sake of clarity, we divide this
section into subsections with each one of them devoted to one term. The tool lemmas in
Section \ref{preliminary} will be used extensively here.

\subsection{Estimate on $\|\la\na\ra\pp_{xx}\psi\|_\infty$}
Using the Duhamel formula, namely (\ref{Inhomoeqs}),
\begin{align*}
\psi(t,\,x,y)=K_0(t,\,\pp_x)\psi_0+K_1(t,\,\pp_x)(\frac12\psi_0+\psi_1)+\int_1^tK_1(t-s,\,\pp_x)F_0(s)\,ds.
\end{align*}
For notational convenience, we may sometimes write $K_0(t)$ for $K_0(t,\,\pp_x)$ and $K_1(t)$ for $K_1(t,\,\pp_x)$.  Therefore,
\begin{align*}
\|\la\na\ra\pp_{xx} \psi\|_\infty\lesssim\,& \|\la\na\ra\pp_{xx} K_0(t )\psi_0\|_\infty
+ \|\la\na\ra\pp_{xx} K_1(t )(\frac12\psi_0+\psi_1)\|_\infty
\\
&+ \|\int_1^t\la\na\ra\pp_{xx} K_1(t-s)F_0(s)ds\|_\infty.
\end{align*}
By Corollary \ref{cor:tool12} and Lemma \ref{lem:estK0K1},
\begin{align*}
& \|\la\na\ra\pp_{xx} K_0(t)\psi_0\|_\infty\\
\lesssim \,&\big(\|\widehat{\pp_{xx}K_0}(t,\,\xi)\|_{L^1_{\xi}(|\xi|\le \frac12)}+\|\widehat{K_0}(t,\,\xi)\|_{L^\infty_{\xi}(|\xi|\geq \frac12)}\big)\, \|\la\na\ra^{2+\va}\pp_{xx}\partial_y\psi_0\|_{L^1_{xy}}
\\
\lesssim  \,&\big(t^{-\frac32}+e^{-t}\big)\|\la\na\ra^{5+\va}\psi_0\|_{L^1_{xy}}\lesssim t^{-\frac32}\|\la\na\ra^{5+\va}\psi_0\|_{X_0}.
\end{align*}
Since the estimates for $K_0$ and $K_1$ are the same, we also have
\begin{align*}
\|\la\na\ra\pp_{xx} K_1(t)(\frac12\psi_0+\psi_1)\|_\infty\lesssim \,t^{-\frac32}\big\|\la\na\ra^{5+\va}(\frac12\psi_0+\psi_1)\big\|_{X_0}.
\end{align*}
Moreover,
\begin{align}
&\left\|\int_1^t\la\na\ra\pp_{xx} K_1(t-s)\,F_0(s)\,ds\right\|_\infty
\notag \\
\lesssim \,&
\int^t_1\|\pp_{xx} K_1(t-s)\,\la\na\ra F_0(s)\|_\infty\,ds \notag\\
\lesssim &  \int^{\frac t2}_1\|\pp_{xx} K_1(t-s)\,\la\na\ra F_0(s)\|_\infty\,ds+
\int^t_{\frac t2}\|\pp_{x} K_1(t-s)\,\la\na\ra \pp_{x} F_0(s)\|_\infty\,ds. \label{attention}
\end{align}
By Corollary \ref{cor:tool12}, Lemma \ref{lem:estK0K1} and Lemma \ref{F0bd},
\begin{align*}
& \int^{\frac t2}_1\|\pp_{xx} K_1(t-s)\,\la\na\ra F_0(s)\|_\infty\,ds
\\
\lesssim \,& \int^{\frac t2}_1(\|\widehat{\pp_{xx} K_1}(t-s,\,\xi)\|_{L^1_{\xi}(|\xi|\le \frac12)}+\|\widehat{K_1}(t-s,\,\xi)\|_{L^\infty_{\xi}(|\xi|\geq \frac12)})\|\na\la\na\ra^{4+\va} F_0(s)\|_{L^1_{xy}}\,ds
\\
\lesssim \,& \int^{\frac t2}_1\langle t-s\rangle^{-\frac32}\|\,\na\la\na\ra^{4+\va} F_0(s)\|_{L^1_{xy}}\,ds
\\
\lesssim \,& \int^{\frac t2}_1\langle t-s\rangle^{-\frac32}s^{-\frac32+\va}\,ds\cdot Q\big(\|(\vec u,\, \psi)\|_Y\big)
\\
\lesssim \,&t^{-\frac32}Q\big(\|(\vec u,\, \psi)\|_Y\big);
\end{align*}
and
\begin{align*}
& \int_{\frac t2}^t\big\|\pp_{x} K_1(t-s)\,\la\na\ra \pp_{x} F_0(s)\big\|_\infty\,ds
\\
\lesssim \,& \int_{\frac t2}^t\big(\big\|\widehat{\pp_{x} K_1}(t-s,\,\xi)\big\|_{L^1_{\xi}(|\xi|\le \frac12)}+\big\|\widehat{K_1}(t-s,\,\xi)\big\|_{L^\infty_{\xi}(|\xi|\geq \frac12)}\big)\|\la\na\ra^{3+\va}\na\partial_x F_0(s)\|_{L^1_{xy}}\,ds
\\
\lesssim \,& \int_{\frac t2}^t\langle t-s\rangle^{-1}s^{-\frac32-\varepsilon}\,ds\, Q\big(\|(\vec u,\, \psi)\|_Y\big)
\\
\lesssim \,&t^{-\frac32}\, Q\big(\|(\vec u,\, \psi)\|_Y\big).
\end{align*}

Combining the estimates above, we obtain
$$
\sup_{t\geq 1}\Big(t^{\frac32}\big\|\la\na\ra\pp_{xx}\psi(t)\big\|_{\infty}\Big)
\lesssim
\|(\vec u_0,\,\psi_0)\|_{X_0} + Q\big(\|(\vec u,\, \psi)\|_Y\big).
$$

\subsection{Estimate of $\|\la\na\ra \pp_t\vec{u}\|_\infty$}
Using the Duhamel formula, we obtain
\begin{align*}
\pp_t u(t,\,x)=\dot K_0(t) u_0 +\dot K_1(t)(\frac12u_0+u_1)+\int_1^t\dot K_1(t-s)\,F_1(s)\,ds,
\end{align*}
then we have
\begin{align*}
\|\pp_t \la\na\ra u(t)\|_\infty
\le&\|\dot K_0(t ) \la\na\ra u_0\|_\infty +\|\dot K_1(t)\la\na\ra(\frac12u_0+u_1)\|_\infty\\
&\quad +\|\int_1^t\dot K_1(t-s)\,\la\na\ra F_1(s)\,ds\|_\infty.
\end{align*}
For the linear parts, by Lemmas \ref{lem:estK0K1}, \ref{lem:tool1} and  \ref{lem:tool2},  we have
\begin{align*}
\|\dot K_0(t ) \la\na\ra u_0\|_{L^\infty_\xi}
\lesssim & \big\|\widehat{\dot K_0}(t, \xi)\big\|_{L^1_\xi(|\xi|\le \frac12)}\big\|\partial_y \la\na\ra u_0\big\|_{L^1_{xy}}\\
&\quad
+\big\|\la \xi \ra^{-1}\widehat{\dot K_0}(t, \xi)\big\|_{L^\infty_\xi(|\xi|\ge \frac12)}\big\|\partial_y \la\na\ra^{3+\epsilon} u_0\big\|_{L^1_{xy}}\\
\lesssim & t^{-\frac32}\,
\big\|\la\na\ra^{4+\epsilon} u_0\big\|_{L^1_{xy}}.
\end{align*}
Also we have,
\begin{align*}
\|\dot K_1(t)\la\na\ra(\frac12u_0+u_1)\|_{L^\infty_{xy}}\lesssim t^{-\frac32}\,
\big\|\la\na\ra^{3+\epsilon} (\frac12u_0+u_1)\big\|_{L^1_{xy}}.
\end{align*}
For the nonlinear part,  by Lemma \ref{lem:claim1} we have
\begin{align*}
\|\int_1^t\dot K_1(t-s)&\la\na\ra F_1(s)\,ds\|_{\infty}
\lesssim
\int_1^t\Big(\|\widehat{\dot K_1}(t-s,\,\xi)\|_{L^1_{\xi}(|\xi|\le \frac12)}\\
&
+\|\widehat{\dot K_1}(t-s,\,\xi)\|_{L^\infty_{\xi}(|\xi|\geq \frac12)}\Big)
\cdot \|\lr^{2+\va}\na F_1(s)\|_{L^1_{xy}} \,ds
\\
\lesssim &\int_1^t\langle t-s\rangle^{-\frac32}\,s^{-\frac32-\varepsilon}ds\cdot Q(\|(\vec u,\,\psi)\|_Y)
\\
\lesssim &t^{-\frac32} Q(\|(\vec u,\,\psi)\|_Y).
\end{align*}
Combining with the estimates above, we deduce that
$$
\sup_{t\geq 1}\Big(t^{\frac32}\big\|\la\na\ra\pp_{t}u(t)\big\|_\infty\Big)
\lesssim
\|(\vec u_0,\,\psi_0)\|_{X_0}+
Q\big(\|(\vec u,\, \psi)\|_Y\big).
$$
Similarly, by the Duhamel formula, we have
\begin{align*}
\pp_t v(t,\,x)=\dot K_0(t) v_0 +\dot K_1(t)(\frac12v_0+v_1)+\int_1^t\dot K_1(t-s)\,F_2(s)\,ds.
\end{align*}
By using Lemma \ref{lem:claim2} instead, it obeys a similar estimate as the one for $u$. Therefore,
$$
\sup_{t\geq 1}\Big(t^{\frac32}\big\|\la\na\ra\pp_{t} \vec u(t)\big\|_\infty\Big)
\lesssim
\|(\vec u_0,\,\psi_0)\|_{X_0}+
Q\big(\|(\vec u,\, \psi)\|_Y\big).
$$

\subsection{Estimate of $\|\pp_x \la\na\ra \vec{u}(t)\|_\infty$}
The estimate for $\|\pp_x \la\na\ra u(t)\|_\infty$ is similar as that on $\|\pp_t \la\na\ra u(t)\|_\infty$,
\begin{align*}
\|\pp_x \la\na\ra u(t)\|_\infty\lesssim \,&\|\pp_x K_0(t ) \la\na\ra u_0\|_\infty +\|\pp_x K_1(t)\la\na\ra(\frac12u_0+u_1)\|_\infty\\
&
+\|\int_1^t \pp_x  K_1(t-s)\, \la\na\ra F_1(s)\,ds\|_\infty
\\
\lesssim &t^{-1}(\|\lr^{3+\va}u_0\|_{L^1_{xy}}+\|\lr^{3+\va}u_1\|_{L^1_{xy}})\\
&+\int_1^t\langle t-s\rangle^{-\frac32}\,s^{-\frac32-\varepsilon}ds\cdot Q(\|(\vec u,\,\psi)\|_Y)
\\
\lesssim& t^{-1}\big( \|(\vec u_0,\,\psi_0)\|_{X_0}+Q\big(\|(\vec u,\, \psi)\|_Y\big)\big).
\end{align*}
$\|\pp_x \la\na\ra v(t)\|_\infty$ can be bounded similarly as the one for $\|\pp_x \la\na\ra u(t)\|_\infty$. We omit the details. Thus,
$$
\sup_{t\geq 1}\Big(t\,\big\|\la\na\ra\pp_{x} \vec u(t)\big\|_\infty\Big)
\lesssim
\|(\vec u_0,\,\psi_0)\|_{X_0}+
Q\big(\|(\vec u,\, \psi)\|_Y\big).
$$

\vskip .3in
\section{Estimates on $\|\lr^3\psi\|_2,\, \|\lr^3\pp_x^2\psi\|_2,\,\|\pp_x^3\psi\|_2,\,\|\lr\pp_t \vec{u}\|_2$ and $\|\partial_x\partial_t v\|_2$}\label{s3}

\vskip .1in
The estimates on $\|\lr^3\psi\|_2,\, \|\lr^2\pp_x^2\psi\|_2,\,\|\pp_x^3\psi\|_2,\,\|\lr\pp_t \vec{u}\|_2$ and $\|\partial_x\partial_t v\|_2$ can be similarly obtained  as in section \ref{s2}.

\vskip .1in
By Lemmas \ref{lem:estK0K1}, \ref{lem:tool3} and \ref{lem:tool4},
\begin{align*}
\|\lr^3\psi\|_2\lesssim\,& \|K_0(t)\lr^3\psi_0\|_2+\|K_1(t)\lr^3(\frac12\psi_0+\psi_1)\|_2
\\
&+\|\int_1^tK_1(t-s)\lr^3F_0(s)\,ds\|_{L^2}
\\
\lesssim \,&t^{-\frac14} (\|\lr^{4+\va}u_0\|_{L^1_{xy}}+\|\lr^{4+\va}u_1\|_{L^1_{xy}})+ \int_1^t(\|\widehat{K_1}(t-s,\,\xi)\|_{L^2_{\xi}(|\xi|\le \frac12)}\\
&+\|\widehat{ K_1}(t-s,\,\xi)\|_{L^\infty_{\xi}(|\xi|\geq \frac12)})\|\,|\na|^{\frac12-\va}\la\na\ra^{\frac72+2\va} F_0(s)\|_{L^1_{xy}}\,ds
\\
\lesssim &t^{-\frac14} \|(\vec u_0,\,\psi_0)\|_{X_0}+\int_1^t\langle t-s\rangle^{-\frac14}\,s^{-1-\va}ds\cdot Q(\|(\vec u,\,\psi)\|_Y)
\\
\lesssim &t^{-\frac14}( \|(\vec u_0,\,\psi_0)\|_{X_0}+Q(\|(\vec u,\,\psi)\|_Y)).
\end{align*}
We define a Fourier multiplier operator $|\pp_{x}|^{1-2\varepsilon}$ as
$$
|\pp_{x}|^{1-2\varepsilon} f(x,y) = \int e^{ix \xi + iy \eta} \, |\xi|^{1-2\varepsilon} \widehat{f}(\xi, \eta)\,d\xi\,d\eta.
$$
Then as before, by Lemmas \ref{lem:estK0K1}, \ref{lem:tool3} and \ref{lem:tool4}, we have
\begin{align*}
\|\lr^3\pp_x^2\psi\|_2\lesssim\,& \|\pp_x^2K_0(t)\lr^3\psi_0\|_2+\|\pp_x^2K_1(t)\lr^3(\frac12\psi_0+\psi_1)\psi_0\|_2
\\
& +\|\int_1^t\pp_x^2K_1(t-s)\lr^2F_0(s)\,ds\|_{L^2}
\\
\lesssim \,& t^{-\frac54}(\|\lr^{6+\va}u_0\|_{L^1_{xy}}+\|\lr^{6+\va}u_1\|_{L^1_{xy}})+\int^\frac t2_1\|\pp_x^2K_1(t-s)\lr^2F_0(s)\|_{L^2}\,ds
\\
&+\int^t_\frac t2\big\||\pp_x|^{\frac32-2\varepsilon}K_1(t-s)\lr^2|\pp_x|^{\frac12+2\varepsilon}F_0(s)\big\|_{L^2}\,ds
\\
\lesssim \,& t^{-\frac54} \|(\vec u_0,\,\psi_0)\|_{X_0}+\int^\frac t2_1\Big(\|\widehat{\pp_x^2K_1}(t-s,\,\xi)\|_{L^2_{\xi}(|\xi|\le \frac12)}\\
&+\|\widehat{ K_1}(t-s,\,\xi)\|_{L^\infty_{\xi}(|\xi|\geq \frac12)}\Big)\|\,|\na|^{\frac12-\va}\la\na\ra^{\frac{11}2+2\va} F_0(s)\|_{L^1_{xy}}\,ds
\\
&+\int^t_\frac t2\Big(\big\||\xi|^{\frac32-2\varepsilon}\widehat{K_1}(t-s,\,\xi)\big\|_{L^2_{\xi}(|\xi|\le \frac12)}\\
&+\big\|\widehat{ K_1}(t-s,\,\xi)\big\|_{L^\infty_{\xi}(|\xi|\geq \frac12)}\Big)\|\,\la\na\ra^{5+2\va}\nabla F_0(s)\|_{L^1_{xy}}\,ds
\\
\lesssim &t^{-\frac54} \|(\vec u_0,\,\psi_0)\|_{X_0}+\int^\frac t2_1\langle t-s\rangle^{-\frac54}\,s^{-1-\va}ds\cdot Q\big(\|(\vec u,\, \psi)\|_Y\big)
\\
&+\int^t_\frac t2\langle t-s\rangle^{-1+\va}\,s^{-\frac32+\va}ds\cdot Q\big(\|(\vec u,\, \psi)\|_Y\big)
\\
\lesssim &t^{-\frac54}\big( \|(\vec u_0,\,\psi_0)\|_{X_0}+Q\big(\|(\vec u,\, \psi)\|_Y\big)\big).
\end{align*}
Now we consider $\|\pp_x^3\psi\|_2$. By Lemmas \ref{lem:estK0K1}, \ref{lem:tool3} and \ref{lem:tool4},
\begin{align*}
\|\pp_x^3\psi\|_2\lesssim\,& \|\pp_x^3K_0(t)\psi_0\|_2+\|\pp_x^3K_1(t)(\frac12\psi_0+\psi_1)\|_2
\\
&+\|\int_1^t\pp_x^3K_1(t-s)F_0(s)\,ds\|_{L^2}
\\
\lesssim \,& t^{-\frac74}(\|\lr^{4+\va}u_0\|_{L^1_{xy}}+\|\lr^{4+\va}u_1\|_{L^1_{xy}})+\int^\frac t2_1\|\pp_x^3K_1(t-s)F_0(s)\|_{L^2}\,ds
\\
&+\int^t_\frac t2\| \,|\pp_x|^{\frac32-2\va}K_1(t-s)|\pp_x|^{\frac32+2\va}F_0(s)\|_{L^2}\,ds.
\end{align*}
We proceed in the same as in the previous estimate,
\begin{align*}
\|\pp_x^3\psi\|_2\lesssim\,& t^{-\frac74} \|(\vec u_0,\,\psi_0)\|_{X_0}
+\int^\frac t2_1\langle t-s\rangle^{-\frac74}\,\big\|\,|\na|^{\frac12-\va}\lr^{\frac72+2\varepsilon}F_0(s)\big\|_{L^1_{xy}}ds
\\
&+\int^t_\frac t2\langle t-s\rangle^{-1+\va}\,\|\,\na\pp_x\lr^{2+2\va}F_0\|_{L^1_{xy}}ds
\\
\lesssim&
t^{-\frac74} \|(\vec u_0,\,\psi_0)\|_{X_0}+\int^\frac t2_1\langle t-s\rangle^{-\frac74}\,s^{-1-\va}ds\cdot Q\big(\|(\vec u,\, \psi)\|_Y\big)
\\
&+\int^t_\frac t2\langle t-s\rangle^{-1+\va}\,s^{-\frac32-\va}ds\cdot Q\big(\|(\vec u,\, \psi)\|_Y\big)
\\
\lesssim &t^{-\frac32}\big(  \|(\vec u_0,\,\psi_0)\|_{X_0}+Q\big(\|(\vec u,\, \psi)\|_Y\big)\big).
\end{align*}
We now bound $\|\langle\nabla\rangle\pp_t \vec{u}\|_2$. By the Duhamel formula,
\begin{align*}
\|\langle\nabla\rangle\pp_t u\|_2\lesssim\,&\|\dot K_0(t ) \langle\nabla\rangle u_0\|_2 +\|\dot K_1(t)\langle\nabla\rangle(\frac12u_0+u_1)\|_2\\
& + \big\|\int_1^t\dot K_1(t-s)\, \langle\nabla\rangle F_1(s)\,ds\big\|_2.
\end{align*}
By Lemma \ref{lem:estK0K1},
\begin{align*}
&\|\dot K_0(t ) \langle\nabla\rangle u_0\|_2 +\|\dot K_1(t)\langle\nabla\rangle(\frac12u_0+u_1)\|_2\\
&\qquad\qquad \lesssim\, t^{-\frac54}(\|\lr^{3+\va}u_0\|_{L^1_{xy}}+\|\lr^{2+\va}u_1\|_{L^1_{xy}}).
\end{align*}
By Lemmas \ref{lem:tool3}, \ref{lem:tool4} and \ref{lem:claim1},
\begin{align*}
& \big\|\int_1^t\dot K_1(t-s)\, \langle\nabla\rangle F_1(s)\,ds\big\|_2 \\
\lesssim & \int_1^{t} \Big(\|\widehat{\dot K_1}(t-s)\|_{L^2_\xi(|\xi|\le \frac12)}
+ \|\widehat{\dot K_1}(t-s)\|_{L^\infty_\xi(|\xi|\ge \frac12)}\Big)\big\||\nabla|^{\frac12-\va}\langle\nabla\rangle^{\frac32+\va} F_1(s)\big\|_2\\
\lesssim& \int_1^{t} \langle t-s\rangle^{-\frac54}\,s^{-\frac32-\va}\,ds \,Q\big(\|(\vec u,\, \psi)\|_Y\big)\\
\lesssim&  t^{-\frac54}\,Q\big(\|(\vec u,\, \psi)\|_Y\big).
\end{align*}
Therefore,
\begin{align*}
\|\langle\nabla\rangle\pp_t u\|_2
\lesssim t^{-\frac54}\big(  \|(\vec u_0,\,\psi_0)\|_{X_0}+Q\big(\|(\vec u,\, \psi)\|_Y\big)\big).
\end{align*}

Similarly,
\begin{align*}
\|\langle\nabla\rangle\pp_t v\|_2\lesssim\,&\|\dot K_0(t ) \langle\nabla\rangle v_0\|_2 +\|\dot K_1(t)\langle\nabla\rangle(\frac12v_0+v_1)\|_2\\
& +\big\|\int_1^t\dot K_1(t-s)\, \langle\nabla\rangle F_2(s)\,ds\big\|_2\\
\lesssim &t^{-\frac54}\big(  \|(\vec u_0,\,\psi_0)\|_{X_0}+Q(\|(\vec u,\,\psi)\|_Y)\big).
\end{align*}
Moreover,
\begin{align*}
\|\partial_x\partial_t v\|_2\lesssim\,&\|\dot K_0(t )\partial_x v_0\|_2 +\|\dot K_1(t)\partial_x(\frac12v_0+v_1)\|_2\\
& +\|\int_1^t \partial_x\dot K_1(t-s)\, F_2(s)\,ds\|_2
\\
\lesssim\, &t^{-\frac32}(\|\lr^{3+\va}v_0\|_{L^1_{xy}}+\|\lr^{2+\va}v_1\|_{L^1_{xy}})+\int_1^t t^{-\frac32}\big\|\,|\na|^{\frac12-\va}\lr^{\frac32+2\va} F_2 \big\|_{L^1_{xy}}ds\\
\lesssim &t^{-\frac32}\|(\vec u_0,\,\psi_0)\|_{X_0}+\int_1^t\langle t-s\rangle^{-\frac32}\,s^{-\frac32-\va}ds\cdot Q(\|(\vec u,\,\psi)\|_Y)
\\
\lesssim &t^{-\frac32}\big(  \|(\vec u_0,\,\psi_0)\|_{X_0}+Q(\|(\vec u,\,\psi)\|_Y)\big).
\end{align*}

\vskip .4in

\appendix

\section{}

This appendix serves four purposes.
It gives an explicit representation of $\Pi_1$ and $\Pi_2$, which has been used in the estimates of
$F_0$, $F_1$ and $F_2$. It also provides the proofs of Lemma \ref{lem:Riesz-bound} and Lemma \ref{lem:RR}.
Finally, the properties of the pressure are also given here.

\subsection{Another expression of  $\Pi_j,\,j=1,\,2$}\label{app1}

This subsection writes out each term of $\Pi_1$ and $\Pi_2$ explicitly. $\Pi_1$ and $\Pi_2$ are previously represented
in vector form in (\ref{N1term}) and (\ref{N2term}), respectively. Our key point here is that each term
is written in a way that it possesses as many directives in the $x$-direction as possible. As seen from
Lemma \ref{lem:estK0K1} in Section \ref{preliminary}, the more $x-$derivatives a term has, the faster
it decays in time. This point has played an important role in the estimates of $F_0$, $F_1$ and $F_2$ in
Section \ref{sec:non}.

\begin{align}\label{N1}
\Pi_1=&-u\pp_x u-v\pp_yu+\frac{\pp_{xy}}{\Dd}(u\pp_xv)-\frac{\pp_{xy}}{\Dd}(v\pp_xu)
\notag\\
&+\frac{\pp_{xx}}{\Dd}(u\pp_xu)+\frac{\pp_{xx}}{\Dd}(v\pp_yu)
-\pp_x\psi\pp_{xx}\psi-\pp_x\psi\pp_{yy}\psi
\notag\\
&+\frac{\pp_{xy}}{\Dd}(\pp_y\psi\pp_{xx}\psi)
+\frac{\pp_{yy}}{\Dd}(\pp_y\psi\pp_{xy}\psi)+\frac{\pp_{xx}}{\Dd}(\pp_x\psi\pp_{xx}\psi)+\frac{\pp_{xx}}{\Dd}(\pp_x\psi\pp_{yy}\psi),
\end{align}
and
\begin{align}\label{N2}
\Pi_2=&-u\pp_x v-v\pp_yv+\frac{\pp_{xy}}{\Dd}(u\pp_xu)+\frac{\pp_{xy}}{\Dd}(v\pp_yu)
\notag\\
&+\frac{\pp_{yy}}{\Dd}(u\pp_xv)+\frac{\pp_{yy}}{\Dd}(v\pp_yv)
-\pp_{yy}\psi\pp_y\psi-\pp_{xx}\psi\pp_y\psi
\notag\\
&+\frac{\pp_{yy}}{\Dd}(\pp_{xx}\psi\pp_y\psi)
+\frac{\pp_{yy}}{\Dd}(\pp_{yy}\psi\pp_y\psi)+\frac{\pp_{xy}}{\Dd}(\pp_{xx}\psi\pp_x\psi)+\frac{\pp_{xy}}{\Dd}(\pp_{yy}\psi\pp_x\psi)
\notag\\
=&-\frac{\pp_{xx}}{\Dd}(u\pp_xv)-\frac{\pp_{xx}}{\Dd}(v\pp_yv)+\frac{\pp_{xxy}}{\Dd}(uu)+\frac{\pp_{xyy}}{\Dd}(uv)
-\frac{\pp_{xx}}{\Dd}(\pp_{xx}\psi\pp_y\psi)
\notag\\
&-\frac{\pp_{xx}}{\Dd}(\pp_{yy}\psi\pp_y\psi)+\frac12\frac{\pp_{xxy}}{\Dd}(\pp_{x}\psi\pp_x\psi)+\frac{\pp_{xyy}}{\Dd}(\pp_{x}\psi\pp_y\psi)
-\frac{\pp_{xy}}{\Dd}(\pp_{y}\psi\pp_{xy}\psi)
\notag\\
=&-2\frac{\pp_{xx}}{\Dd}(u\pp_xv)+\pp_x(uv)+2\frac{\pp_{xy}}{\Dd}(u\partial_x u)
\notag\\
&-2\frac{\pp_{xy}}{\Dd}(\pp_{y}\psi\pp_{xy}\psi)+\frac{\pp_{xy}}{\Dd}(\pp_{x}\psi\pp_{xx}\psi)
+\frac{\pp_{yy}}{\Dd}(\pp_y\psi\pp_{xx}\psi)+\frac{\pp_{yy}}{\Dd}(\pp_{x}\psi\pp_{xy}\psi).
\end{align}

\subsection{Proof of Lemma \ref{lem:Riesz-bound}}\label{app2}
It is easily followed by the Littlewood-Paley's decomposition. Indeed,
\begin{align*}
\big\|\mathcal R  f\big\|_{L^1_{xy}}
\le &
\sum\limits_{M\le 1} \big\|P_M \mathcal R f\big\|_{L^1_{xy}}+\sum\limits_{M\ge 1} \big\|P_M\mathcal R f\big\|_{L^1_{xy}}\\
\le &
\sum\limits_{M\le 1} \big\|P_M  f\big\|_{L^1_{xy}}+\sum\limits_{M\ge 1} \big\|P_M f\big\|_{L^1_{xy}}\\
\lesssim &
\sum\limits_{M\le 1} M^{\epsilon}\big\|P_M |\nabla|^{-\epsilon}  f\big\|_{L^1_{xy}}+
\sum\limits_{M\ge 1} M^{-\epsilon}\big\|P_M\langle\nabla\rangle^{\epsilon} f\big\|_{L^1_{xy}}\\
\lesssim &\big\||\nabla|^{-\epsilon}\langle\nabla\rangle^{2\epsilon}f\big\|_{L^1_{xy}}.
\end{align*}
This prove Lemma \ref{lem:Riesz-bound}.

\vskip .1in

\subsection{Proof of Lemma \ref{lem:RR}}\label{app3}
By the Littlewood-Paley decomposition,
\begin{align*}
\Big\||\nabla|^\alpha \langle\nabla\rangle^{N_0}\frac{\partial_x\partial_y}{-\Delta}f\Big\|_{L^1_{xy}}
\lesssim &
\sum\limits_{M\le 1} M^{\alpha-1}\big\|\partial_xf\big\|_{L^1_{xy}}+\sum\limits_{M\ge 1} M^{N_0+\alpha-1}\big\|\partial_xf\big\|_{L^1_{xy}}\\
\lesssim &
\sum\limits_{M\le 1}M^{\frac{\epsilon}{\alpha-\epsilon}}  M^{-1}\big\|\nabla \partial_xf\big\|_{L^1_{xy}}^{\alpha-\epsilon}\big\|\partial_xf\big\|_{L^1_{xy}}^{1-\alpha+\epsilon}\\
& +\sum\limits_{M\ge 1}M^{-\epsilon} \big\|\langle\nabla\rangle^{N_0-1+\alpha+\epsilon}\partial_xf\big\|_{L^1_{xy}}\\
\lesssim &
\sum\limits_{M\le 1}M^{\frac{\epsilon}{\alpha-\epsilon}}\big\| \partial_xf\big\|_{L^1_{xy}}^{\alpha-\epsilon}\big\|f\big\|_{L^1_{xy}}^{1-\alpha+\epsilon}+\big\|\langle\nabla\rangle^{N_0-1+\alpha+\epsilon}\partial_x f\big\|_{L^1_{xy}}\\
\lesssim &
 \|f\|_{L^1_{xy}}^{1-\alpha+\epsilon}\|\partial_x f\|_{L^1_{xy}}^{\alpha-\epsilon}+\big\|\langle\nabla\rangle^{N_0-1+\alpha+\epsilon}\partial_x f\big\|_{L^1_{xy}}.
\end{align*}
This proves Lemma \ref{lem:RR}.

\subsection{Properties of the pressure $P$}\label{app4}
First of all, by applying $\nabla\cdot \vec u=0$, we have
\begin{align*}
P&=\frac{\nabla\cdot\big[\nabla\cdot(\nabla \phi\otimes\nabla\phi+\vec u\otimes \vec u)\big]}{-\Delta}\\
&=-2\partial_y\phi+\frac{\nabla\cdot\big[\nabla\cdot(\nabla \psi\otimes\nabla\psi+\vec u\otimes \vec u)\big]}{-\Delta}.
\end{align*}
Therefore, for any $t\ge 1$, by Sobolev' inequality,
\begin{align*}
\|P(t)\|_{H^N}\lesssim &
\|\partial_y\psi\|_{H^N}+\|\nabla \psi\otimes\nabla\psi\|_{H^N}+\|\vec u\otimes \vec u\|_{H^N}\\
\lesssim &
\|\nabla\psi\|_{H^N}+\|\nabla\psi\|_\infty\|\nabla \psi\|_{H^N}+\|\vec u\|_\infty\|\vec u\|_{H^N}\\
\lesssim & t^\varepsilon \big(\|(\vec u,\psi)\|_Y+\|(\vec u,\psi)\|_Y^2\big)\\
\lesssim & \varepsilon_0 \,t^\varepsilon.
\end{align*}
This implies that $P\in C\big([1,\infty);H^N(\R^2)\big)$.  Furthermore,
\begin{align*}
\|P(t)\|_\infty
\lesssim &
\|\partial_y\psi\|_\infty+\|\langle\nabla\rangle^2(\nabla \psi\otimes\nabla\psi)\|_2+\|\langle\nabla\rangle^2(\vec u\otimes \vec u)\|_2\\
\lesssim &
\|\nabla\psi\|_\infty+\|\langle\nabla\rangle^3\psi\|_2\|\nabla\psi\|_\infty+\|\langle\nabla\rangle^2\vec u\|_2\| \vec u\|_\infty\\
\lesssim & t^{-\frac12} \big(\|(\vec u,\psi)\|_Y+\|(\vec u,\psi)\|_Y^2\big)\\
\lesssim & \varepsilon_0 \,t^{-\frac12}.
\end{align*}

\vskip .4in
\section*{Acknowledgements}
J. Wu was partially supported by NSF grant DMS1209153.
J. Wu thanks Professor Chongsheng Cao for discussions and the School of Mathematical Sciences, Beijing Normal University for its hospitality.
Y. Wu was partially supported by the NSFC (No.11101042).
Xu was partially supported by NSFC (No.11371059), BNSF (No.2112023). Y. Wu and Xu were both partially supported by the Fundamental Research Funds for
the Central Universities of China.

\end{document}